\documentclass[12pt]{amsart}
\usepackage[utf8]{inputenc}
\usepackage{amsthm}
\usepackage{amsmath}
\usepackage{stmaryrd}
\usepackage{amssymb} 
\usepackage{comment}
\usepackage{graphicx}
\usepackage{mathrsfs}
\usepackage{fullpage}
\usepackage{mathtools}
\usepackage{colonequals}
\usepackage{hyperref}
\usepackage[all]{xy}
\usepackage[usenames,dvipsnames]{xcolor}
\usepackage[ left = 1.2in, right = 1.2in, 
             top = 1.2in, bottom = 1.2in ]{geometry}
\usepackage{tikz-cd}

\numberwithin{equation}{subsection} 
\numberwithin{figure}{subsection} 

\makeatletter
\let\c@equation\c@figure
\makeatother
\newcommand{\Aff}{\mathbb{A}}
\newcommand{\C}{\mathbb{C}}

\newcommand{\M}{\mathbb{M}}
\newcommand{\N}{\mathbb{N}}
\newcommand{\PP}{\mathbb{P}}
\newcommand{\Q}{\mathbb{Q}}
\newcommand{\R}{\mathbb{R}}

\newcommand{\Z}{\mathbb{Z}}
\newcommand{\ZZ}{\mathbb{Z}}

\newcommand{\cC}{\mathcal{C}}
\newcommand{\cD}{\mathcal{D}}
\newcommand{\cE}{\mathcal{E}}

\newcommand{\cO}{\mathcal{O}}
\newcommand{\cP}{\mathcal{P}}

\newcommand{\cX}{\mathcal{X}}

\newcommand{\sC}{\mathscr{C}}
\newcommand{\sD}{\mathscr{D}}

\newcommand{\Ga}{\operatorname{G}_a}
\newcommand{\Gm}{\operatorname{G}_m}

\newcommand{\Spec}{\operatorname{Spec}}
\newcommand{\tSigma}{{\tilde{\Sigma}}}
\newcommand{\tP}{{\tilde{P}}}

\DeclareMathOperator{\DLog}{DLog}
\DeclareMathOperator{\End}{End}

\DeclareMathOperator{\Hom}{Hom}

\DeclareMathOperator{\recc}{rec}
\DeclareMathOperator{\Res}{Res}
\DeclareMathOperator{\Trop}{Trop}
\DeclareMathOperator{\trop}{trop}
\DeclareMathOperator{\Span}{Span}
\DeclareMathOperator{\Star}{Star}
\DeclareMathOperator{\Vect}{Vec}

\DeclareMathOperator{\Ext}{Ext}

\newtheorem{theorem}[equation]{Theorem}
\newtheorem{corollary}[equation]{Corollary}
\newtheorem{lemma}[equation]{Lemma}
\newtheorem{prop}[equation]{Proposition}

\theoremstyle{definition}
\newtheorem{rem}[equation]{Remark}
\newtheorem{example}[equation]{Example}

\theoremstyle{definition}
\newtheorem{definition}[equation]{Definition}

\newcommand{\dR}{{\operatorname{dR}}}
\newcommand{\un}{{\operatorname{un}}}
\newcommand*\ps[1]{{\llbracket#1\rrbracket}}
\newcommand\nc{\newcommand}
\nc\on{\operatorname}

\makeatletter
\newcommand{\exterior}[1]{\mathop{\mathpalette\exterior@{#1}}}
\newcommand{\exterior@}[2]{%
  \raisebox{\depth}{%
  \fontsize{\sf@size}{0}%
  \m@th
  $\ifx#1\displaystyle\textstyle\else#1\fi\bigwedge$}%
  ^{\mspace{-2mu}#2}%
  \kern-\scriptspace
}
\makeatother

\title{The Unipotent Tropical Fundamental Group}
\author{Kyle Binder and Eric Katz}
\date{}

\begin{document}

\begin{abstract}
We define the unipotent tropical fundamental group of a polyhedral complex in $\R^n$ as the Tannakian fundamental group of the category of unipotent tropical vector bundles with integrable connection. We show that it is  computable in that it satisfies a Seifert--Van Kampen theorem and has a description for fans in terms of a bar complex. We then review an analogous classical object, the unipotent de Rham fundamental group of a sch\"{o}n subvariety of a toric variety. Our main result is a correspondence theorem between classical and tropical unipotent fundamental groups: there is an isomorphism between the unipotent completion of the fundamental group of a generic fiber of a tropically smooth family over a disc and the tropical unipotent fundamental group of the family's tropicalization. This theorem is established using Kato--Nakayama spaces and a descent argument. It requires a slight enlargement of the relevant categories, making use of enriched structures and partial compactifications. 
\end{abstract}

\maketitle

\section{Introduction}
Tropical geometry attaches a balanced weighted polyhedral complex $\Trop(X)\subseteq\R^n$, called a tropical variety, to a subvariety of an algebraic torus over a field $K$, $X\subseteq 
(K^*)^n$. The complex  captures some of the geometry of $X$, but it is not always obvious how to extract it. One useful tool in this direction is the Tropical Homology of Itenberg--Katzarkov--Mikhalkin--Zharkov \cite{IKMZ} which attaches homology groups to $\Trop(X)$ and, when $K$ is the quotient field of germs of holomorphic functions on a disc and $X$ obeys a tropical smoothness condition, computes the classical homology of a generic fiber of $X$. It is natural to ask if the fundamental group can be computed combinatorially, that is, via $\Trop(X)$. The answer to this question in the most na\"{i}ve phrasing is negative. Indeed, this is the case even when $X=V\cap (\C^*)^n$ for a linear subspace $V$. Then, $\Trop(X)$ is the Bergman fan attached to the matroid of $V$, and it is known by a counterexample due to Rybnikoff \cite{Rybnikov,Bartolo:invariants} that the matroid of $V$ does not determine the fundamental group of $V\cap (\C^*)^n$. For more about the fundamental group of arrangements and their complements, see \cite{Schenck:hyperplanes, Suciu:fundamental}

We ask a more modest version of this question by replacing the fundamental group by its pro-unipotent completion, the {\em unipotent fundamental group}. We can reformulate this approach using the Tannakian formalism. Following Grothendieck's approach \cite{Grothendieck:pi1} which studies the fundamental group through coverings, Deligne \cite{Deligne:groupefondamental} studied the fundamental group through neutral Tannakian categories \cite{Saavedra,DM:tannakian}. Neutral Tannakian categories axiomatically generalize local systems of vector spaces on a topological space. By a  reconstruction theorem, such a category is equivalent to representations of a pro-algebraic {\em fundamental group}. By specializing to unipotent objects, the resulting fundamental group realizes the pro-unipotent completion.  The main goal of this paper is to demonstrate that under a tropical smoothness hypothesis, the unipotent fundamental group of a subvariety of an algebraic torus can be computed tropically.

We begin by introducing the unipotent tropical fundamental group using the category of unipotent tropical vector bundles with integrable connections. A tropical vector bundle is a locally constant sheaf $E$ of $\C$-vector spaces on a rational polyhedral complex $\Sigma$ in $\R^n$ in the poset topology while an integrable connection on $E$ is an element 
\[\theta\in\Gamma(\Sigma,\Omega^1_{\Sigma}\otimes\on{End}(E))\]
(where $\Omega^1_{\Sigma}$ is the sheaf of tropical differential forms) satisfying $\theta\wedge\theta=0$. We show that such a fundamental group can be computed on fans by the bar construction and can be described on complexes by a generalized Seifert--Van Kampen theorem. This fundamental group is much richer than the pro-unipotent completion of the more familiar topological fundamental group of $\Sigma$ in that it incorporates piecewise-linear data coming from the embedding $\Sigma\hookrightarrow \R^n$. 

We relate the unipotent tropical fundamental group to an invariant on the algebraic side, the unipotent log de Rham fundamental group of a log scheme. This is the Tannakian fundamental group of vector bundles equipped with an integrable log connection. Here, ``log'' means that one works with log differential forms instead of regular ones. This is a natural choice, as tropical varieties are closely related to the dual complex of a semistable degeneration of a scheme over a DVR, and the closed fiber of such a degeneration carries a log structure. This fundamental group is not so far removed from classical invariants: it is  isomorphic to the pro-unipotent completion of the usual topological fundamental group of the Kato--Nakayama space, which is, under certain circumstances, homeomorphic to a generic fiber of the semistable family. 

Our main theorem is an isomorphism between the unipotent fundamental group of a fiber of a family of varieties and the unipotent tropical fundamental group of its tropicalization. Here, we work with a {\em tropically smooth} algebraic subvariety $X$ of an algebraic torus $(K^*)^n$ over a valued field $K$ whose residue field is $\C$. Tropical smoothness is a very restrictive condition on $X$. It requires that there is a polyhedral structure on $\Trop(X)$ such that the star of every face looks like a Bergman fan. Tropical smoothness implies that $X$ is sch\"{o}n, a strong smoothness property. Moreover, it allows one to define a semistable family $\cX$ over an extension $\cO'$ of the DVR of $K$. The closed fiber of this semistable family, $X_0$ is a strict normal crossings scheme whose closed strata are compactifications of hyperplane arrangement complements. By employing Kato--Nakayama spaces, when $K$ is the field of germs of meromorphic functions in a neighborhood of the origin (whose valuation is given by vanishing order at the origin), one can relate the unipotent fundamental group of a generic fiber of $\cX$ to the unipotent log de Rham fundamental group of $X_0$. It, in turn is compared to the unipotent tropical fundamental group of $\Trop(X)$. Philosophically, the de Rham fundamental group is sensitive to connections while the tropical fundamental group only sees the residues of those connections; for tropically smooth varieties residues are enough. In that sense, our main theorem is very closely related to the  Orlik--Solomon theorem on the cohomology of hyperplane arrangement complements \cite{OS:Hyperplanes}.  Our main theorem (with notation deferred) is the following:

\begin{theorem} \label{t:maintheoremintro}
Let $K$ be the field of fractions of the ring $\cO$ of germs of holomorphic functions on $\C$ in a neighborhood of the origin. Let $X\subseteq (K^*)^n$ be a tropically smooth subvariety. Let $\Delta$ be a projective rational polyhedral fan supporting the recession fan of $\Trop(X)$, and let $R$ be a set of rays of $\Delta$ contained in the recession fan of $\Trop(X)$. For sufficiently small nonzero $t\in\C$, let $\overline{X}_t$ be the closure in $\PP(\Delta)$ of $X_t\subseteq (\C^*)^n$. Set
\[\mathring{X}_t\coloneqq \overline{X}_t\setminus\left(\bigcup_{\rho \in \Delta_{(1)}\setminus R} D_\rho\right)\]
where $D_\rho$ is the divisor corresponding to $\rho$. 

Let $x_t\in \mathring{X}_t\cap (\C^*)^n$. Then there is a vertex $v$ of $\Sigma$,an attached fiber functor $\omega_{\Sigma,v}$, and an isomorphism
of unipotent and tropical fundamental groups
\[\pi_1^{\un}(\mathring{X}_t,x_t)\cong \pi_1^{\un}((\Trop(X),\Z^n,R),\omega_{\Sigma,v})\]
for any polyhedral structure on $\Trop(X)$.
\end{theorem}

When $R$ is empty, this theorem describes the fundamental group of a generic fiber $X_t\subset (\C^*)^n$.

By employing a homotopy equivalence between $\mathring{X}_t$ and a particular Kato--Nakayama space, our main theorem reduces to a comparison theorem (Proposition~\ref{p:mainmaintheorem}) between the log de Rham fundamental group and tropical fundmanetal group. We prove this comparison theorem by a Seifert--Van Kampen argument. We can decompose both $X_0$ and $\Sigma$ into simpler pieces: strata and stars, respectively. The Seifert--Van Kampen theorem follows from descent statements which on the tropical side are immediate but on the log side require a 
 descent theorem for locally free sheaves on strictly normal crossing schemes in Appendix~\ref{a:appendix}. We then produce an equivalence of Tannakian categories for the de Rham fundamental groups of the strata and the tropical fundamental groups of the stars. Our arguments here are a slight generalization of Kohno's description of the unipotent fundamental group of a hyperplane arrangement complement \cite{Kohno:holonomy}, albeit phrased in more combinatorial language. 

 This proof requires much machinery to keep track of the relevant log structures and differential forms. Indeed,
 the strata in the log setting are given the inverse image log structure which includes a bit more data than the underlying scheme. To incorporate the data on the tropical side, we introduce the notion of enriched fans. Moreover, because we would like to consider only partial compactifications, i.e.,~removing only the intersection with some toric divisors from $\overline{X}$, we must keep track of certain divisors, parameterized by a set of rays. Again, this has to be accounted for on the combinatorial side by a {\em partial compactification} which, in turn, is reflected by differential forms. 
 The language of partial compactifications is suitable for our purposes, but more general language can be found in the work of Amini--Piquerrez \cite{AP:Homology}. Also, we expect that our theorem holds in mixed characteristic although one must replace the log de Rham setting with a log crystalline or log rigid one.

We outline this paper. In Section~\ref{s:tannakianfundamentalgroups}, we review Tannakian fundamental groups and their descent theory, while 
Section~\ref{s:enrichedfans} introduces enriched fans with partial compactifications and tropical vector bundles with connections. The tropical fundamental group is defined, shown to obey a Seifert--Van Kampen theorem, and is computed for fans in Section~\ref{s:tropicalpi1}. Section~\ref{s:logstructures} studies log structures on subschemes of toric varieties and toric schemes. The subject of Section~\ref{s:logpi1} is the unipotent log de Rham fundamental group which is shown to obey descent in the strict normal crossings scheme setting.  Our main theorem is proved locally and then globalized in Section~\ref{s:correspondencetheorem}. Appendix~\ref{a:appendix} proves a necessary descent statement on strict normal crossing schemes.

\subsection{Acknowledgements}

This work arose from a question asked by Daniel Litt. The works of Omid Amini and his collaborators were an inspiration. The authors also benefited from conversations with Dan Abramovich, Piotr Achinger, Brian Conrad, Chris Eur, Matt Satriano, Martin Ulirsch, and Jonathan Wise.
Kyle Binder was partially supported by NSF Grant DMS-1700194 while
Eric Katz was partially supported by NSF Grant DMS-1748837.



\section{Tannakian Fundamental Groups} \label{s:tannakianfundamentalgroups}

In this section, we define the fundamental group of a Tannakian category and recall descent theory and its connection to the generalized Seifert--Van Kampen theorem.

\subsection{Definitions}

Let $K$ be a field. 
A {\em neutral Tannakian category} $\mathscr{C}$ over $K$  (see e.g.,~\cite{Saavedra,Deligne:CT} for definitions)  is an essentially small rigid abelian tensor category with unit object $\mathbf{1}$ satisfying $\on{End}_{\mathscr{C}}(\mathbf{1})=K$ and which possesses a fiber functor $\omega\colon \mathscr{C}\to \on{Vect}_K$ (i.e.,~ a faithful exact $K$-linear tensor functor). For a fiber functor $\omega$ on $\sC$, let $\pi(\sC,\omega)$ denote the Tannakian fundamental group (see, e.g.,~\cite{Deligne:groupefondamental,Szamuely}), that is, the pro-algebraic group isomorphic to the group $\on{Aut}^\otimes(\omega)$ of tensor automorphisms of $\omega$. Then, $\sC$ is equivalent to the category of representations of $\pi(\sC,\omega)$ such that $\omega$ corresponds to the forgetful functor taking representations to the underlying vector space. Recall that an equivalence of categories $\sC$ and $\sD$ is a pair of functors $F\colon \sC\to\sD$ and $G\colon \sD\to \sC$ together with isomorphisms of functors $\varepsilon_{\sC} \colon G\circ F\to \on{Id}_{\sC}$ and $\varepsilon_{\sD} \colon F\circ G\to \on{Id}_{\sD}$. It is well-known that $\sC$ and $\sD$ are equivalent if and only if there is a fully faithful and essentially surjective functor $F\colon \sC\to\sD$. 

Given a neutral Tannakian category $\sC$ with unit object $\mathbf{1}$, we say that an object $E$ of $\sC$ is {\em unipotent} if there is a filtration
\[E=E^0\supset E^1\supset \dots\supset E^{n+1}=0.\]
whose associated gradeds are direct sums of $\mathbf{1}$.
Write $\sC^{\un}$ for the full subcategory of unipotent objects in $\sC$. It is also a neutral Tannakian category, and we may write $\pi^{\un}(\sC,\omega)$ for $\pi(\sC^{\un},\omega)$,
as $ \pi(\sC^{\un}, \omega) $ is isomorphic to the pro-unipotent completion of $\pi(\sC, \omega) $ \cite{vezzani2012pro}.

\subsection{Descent categories}

In this section, we review descent categories, a natural framework for the Seifert--Van Kampen theorem for Tannakian fundamental groups. They will be used to globalize a special case of our main theorem.  Here, we will follow \cite{Stix:SvK} which is phrased in terms of Galois, rather than Tannakian categories.

We  will use the  language of categories over simplicial complexes. Let $\Sigma$ be a connected two-dimensional simplicial complex.  A {\em neutral Tannakian category} $\sC_{\Sigma}$ over $\Sigma$ is the data of a neutral Tannakian category $\sC_F$ for each face $F\in \Sigma$ together with an exact $K$-linear faithful tensor functor $i_{F_1,F_2}\colon \sC_{F_1}\to \sC_{F_2}$ for each inclusion of faces $F_1\subset F_2$ obeying the cocycle condition $i_{F_1,F_3}=i_{F_2,F_3}\circ i_{F_1,F_2}$ for $F_1\subset F_2\subset F_3$.

Such neutral Tannakian categories arise by considering $\sD$-spaces over $\Sigma$ where $\sD$ is a category of ``spaces,'' for example, schemes. A $\sD$-space over $\Sigma$ is the data of an object $X_F$ in $\sD$ for each face $F\in \Sigma$ together with a morphism $j_{F_1,F_2}\colon X_{F_2}\to X_{F_1}$ for each inclusion of faces $F_1\subset F_2$ (satisfying the cocycle condition). For example, given a covering of a connected space by open sets $\{U_v\}$, we can let $\Sigma$ be the $2$-skeleton of the  nerve of the covering and set $U_F=\bigcap_{v\in F} U_v$.
Given some type of ``local system'' (i.e.,~local systems of vector spaces, vector bundles with integrable connections, etc.) and a $\sD$-space $(X_F,j_{F_1,F_2})$, one can define a neutral Tannakian category $\sD_{\Sigma}$ over $\Sigma$ by
\begin{enumerate}
\item letting $\sC_F$ be the category of ``local systems'' over $X_F$ equipped with some fiber functor $\omega_F$, and 
\item setting $i_{F_1,F_2}=j_{F_1,F_2}^*$, provided that it is faithful.
\end{enumerate}
Often, the condition $\on{End}_{\mathscr{C}}(\mathbf{1})=K$ will require the connectedness of $X_F$.

\begin{definition}
  The {\em descent category} $\sC(\Sigma)$ attached to $\sC_\Sigma$, a Tannakian category over $\Sigma$, is the category whose objects are tuples  $((\xi_v),(\phi_{e}))$ where
  \begin{enumerate}
      \item for each vertex $v\in \Sigma_0$, $\xi_v\in \on{Obj} \sC_v$,
      \item for each directed edge $e=vw$, an isomorphism
    \[\phi_e\colon i_{v,e}(\xi_v)\to i_{w,e}(\xi_w)\]
    in $\sC_e$ such that $\phi_{\overline{e}}=\phi_e^{-1}$ and for all $2$-faces $F=uvw$, we have the following equality in $\sC_F$:
    \[i_{uw,F}(\phi_{uw})=i_{vw,F}(\phi_{vw})\circ i_{uv,F}(\phi_{uv}).\]
  \end{enumerate}
  A morphism $\psi\colon ((\xi_v),(\phi_{e}))\to ((\xi'_v),(\phi'_{e}))$ in $\sC(\Sigma)$ is a collection of morphisms $\psi_v\colon \xi_v\to \xi'_v$ in $\sC_v$ such that for each edge $e=vw$, the following diagram commutes
  \[\xymatrix{
  i_{v,e}(\xi_v)\ar[d]^{\phi_e}\ar[r]^{i_{v,e}(\psi_v)}& i_{v,e}(\xi'_v)\ar[d]^{\phi'_e}\\
  i_{w,e}(\xi_w)\ar[r]^{i_{w,e}(\psi_w)}& i_{w,e}(\xi'_w).
  }\]
\end{definition}

\begin{prop} For a face $F$ of $\Sigma$, a vertex $v$ of $F$, and a fiber functor $\omega_F$ on $\sC_F$, the fiber functor
\[\omega\colon \sC(\Sigma)\to \on{Vec}\]
given by $\omega((\xi_v),(\phi_e))=\omega_F(i_{v,F}(\xi_v))$ makes $\sC(\Sigma)$ into a neutral Tannakian category.
\end{prop}

\begin{proof}
We only need to verify faithfulness of $\omega$. If a morphism 
\[\psi\colon ((\xi_v),(\phi_{e}))\to ((\xi'_v),(\phi'_{e}))\]
maps to $0$ under $\omega_F$, then by faithfulness of $\omega_F$ and $i_{v,F}$, $\psi_v=0$. For any edge $e=vw$, $i_{v,e}(\psi_v)=0$. By the faithfulness of $i_{w,e}$, $\psi_w=0$. By travelling along the connected edge-graph of $\Sigma$, we conclude $\psi_{v'}=0$ for all vertices $v'$.
\end{proof}

A functor $A\colon \sC_\Sigma\to \sD_\Sigma$ between neutral Tannakian categories over $\Sigma$ is a collection of functors $\{A_F\}_{F\in\Sigma}$ such that $A_F\colon\sC_F\to \sD_F$ is compatible with restriction maps in the following sense: for any inclusion of faces $F_1\subset F_2$, the following diagram commutes
\[\xymatrix{
\sC_{F_1}\ar[d]\ar[r]^{A_{F_1}}&\sD_{F_1}\ar[d]\\
\sC_{F_2}\ar[r]^{A_{F_2}}&\sD_{F_2}.}\]

Such a functor $A\colon \sC_\Sigma\to \sD_\Sigma$ naturally induces a functor $A\colon \sC(\Sigma)\to \sD(\Sigma)$ between descent categories. The following is a straightforward verification:
\begin{lemma} \label{l:descentequivalence} Let $A\colon\sC_{\Sigma}\to \sD_{\Sigma}$ be a functor between Tannakian categories over $\Sigma$ such that
\begin{enumerate}
  \item for all vertices $v\in\Sigma$, $A_v$ is an equivalence of categories,
  \item for all edges $e\in\Sigma$, $A_e$ is fully faithful, and
  \item for all $2$-dimensional faces $F\in\Sigma$, $A_F$ is faithful.
\end{enumerate}
Then, the functor $A\colon \sC(\Sigma)\to \sD(\Sigma)$ is an equivalence of categories.
\end{lemma}

 We may attach group data to a neutral Tannakian category $\sC$ over $\Sigma$ 
 (together with paths between the restrictions of fiber functors to faces) by taking the Tannakian fundamental groups $\pi(\sC_F,\omega_F)$. This gives a presentation of $\pi(\sC_\Sigma,\omega)$ as a subgroup of the amalgamation of the groups $ \pi(\sC_{v}, \omega_{v})$, generalizing the more familiar graph of groups construction \cite{Serre:trees,Oda:ramification2}. See, e.g.,~\cite{Stix:SvK} for  details. If each $\sC_F$ is a unipotent category, to obtain the fundamental group of $\sC_\Sigma^\un$, it will be necessary to pro-unipotently complete the group $\pi(\sC_\Sigma,\omega)$.

\section{Enriched Fans and Tropical Bundles} \label{s:enrichedfans}

\subsection{Enriched Fans}

We will use a slight extension of the category of rational polyhedral fans and complexes. Our definitions are motivated by toric and logarithmic geometry. 
Rational polyhedral fans and complexes correspond to toric varieties and schemes, respectively, and vectors in the complexification of the dual lattice correspond to log differentials on the toric schemes with the usual log structure. However, we will need to impose some slightly unusual log structures and  log differentials on these schemes necessitating the concept of {\em enriched fans and complexes with partial compactifications}.

We will let $M$ be a finitely generated free abelian group and $N=M^\vee\coloneqq \Hom(M,\Z)$.  Write $N_\R$ (resp.,~$N_{\C}$) for $N\otimes\R$ (resp.,~$N\otimes\C$). We view $N$ as a lattice in $N_\R$.

\begin{definition} 
A {\em rational polyhedron} in $N_\R$ is the intersection of finitely many half-spaces of the form
$ \langle x, m \rangle \geq a $ for some $m \in M$ and $ a \in \mathbb{R} $.
The faces of such a polyhedron are obtained by replacing some of the 
inequalities with equalities.  It is said to be a {\em rational cone} if each $a$ is equal to $0$. The relative interior $\mathring{P}$ of a rational polyhedron $P$ is the topological interior of $P$, considered as a subspace of its affine span.

For a rational polyhedron $P$ given by
\[P\coloneqq\{x\in N_\R| \langle x,m_i\rangle \geq a_i\}\]
for $m_i\in M$ and $a_i\in\Q$, the {\em recession cone} $\recc(P)$ of $P$ is given by
\[\recc(P)\coloneqq \{x\in N_\R| \langle x,m_i\rangle \geq 0\}.\]
\end{definition}


\begin{definition}
A {\em rational polyhedral complex} $\Sigma$ in $N_\R$ is a collection of rational polyhedra in $N_\R$ that contains each of its faces and for which the non-empty intersection of two polyhedra is a face of each. For a polyhedron $P$ of $\Sigma$, write $P\in\Sigma$. It is {\em complete} if the union of its polyhedra is $N_\R$, and it is {\em completable} if it is a subcomplex of a complete polyhedral complex. Write $\Sigma_{(k)}$ for the polyhedra of $\Sigma$ of dimension $k$. For a set $X\subseteq N_\R$, we say $\Sigma$ supports $X$ (alt.,~$X$ is supported on $\Sigma$) if $X$ is a union of polyhedra of $\Sigma$. The {\em support} of $\Sigma$ is the union $|\Sigma|$ of the polyhedra of $\Sigma$. 
For a polyhedron $P$, write $N_{P,\R}=\Span_{\R}(P-w)$ and $N_P=N\cap N_{P,\R}$ for the {\em tangent space} and {\em tangent lattice} of $P$ where $w\in \mathring{P}$ is a point in the relative interior of $P$.
For a point $x\in|\Sigma|$, let the {\em tangent cone of $\Sigma$ at $x$} be defined by
\[T_x\Sigma\coloneqq \{v\in N_{\R}| \text{for all sufficiently small } \varepsilon>0, x+\varepsilon v\in|\Sigma|\}.\]
A {\em rational fan} is a rational polyhedral complex, each of whose members is a rational cone. It is {\em pointed} if the origin is a cone of it. Henceforth, all fans will be pointed.
A rational fan is {\em unimodular} if all of its cones are spanned by a subset of a basis of the lattice $N$. 
It is {\em projective} if it is the normal fan of an integral polytope. It is {\em quasi-projective} if it is the subfan of a projective fan.
The {\em open star} $\Star^\circ_P(\Sigma)$ of a polyhedron $P$ in a rational polyhedral complex $\Sigma$ is the union of the relative interiors of polyhedra of $\Sigma$ containing $P$:
$\Star^\circ_P(\Sigma)\coloneqq \bigcup_{Q\supseteq P} \mathring{Q}.$
The {\em poset topology} on the set $|\Sigma|$ underlying a rational polyhedral complex $\Sigma$ is the topology whose closed sets are the unions of closed polyhedra of $\Sigma$. Consequently, a basis for the topology is given by open stars of faces.
\end{definition}

By \cite{BGS}, the recession cones of a completable rational polyhedral complex form a rational fan.

\begin{definition} \label{d:starquotientfan}
Let $\Delta$ be a fan, and let $\tau$ be a cone of $\Delta$. Define the {\em star-quotient fan} $\Delta_\tau$ to be the fan in $N_{\R}/N_{\tau,\R}$ given by   
\[\Delta_{\tau}\coloneqq \bigcup_{\sigma\supseteq \tau} (\sigma+N_{\tau,\R})/N_{\tau,\R}.\]
When $ \tau $ is understood, we will write $\overline{\sigma}=(\sigma+N_{\tau,\R})/N_{\tau,\R}$ for the cone in $\Delta_\tau$ corresponding to $\sigma\in\Delta$.
\end{definition}

\begin{definition} An {\em enriched fan} $(\Delta,\pi\colon N\to N')$ is the following data:
\begin{enumerate}
\item a rational fan $\Delta$ in $N'_\R$; and
\item a surjection of finitely generated free abelian groups $\pi\colon N\to N'$ with dual inclusion $\pi^*\colon M'\hookrightarrow M$ (as a saturated sub-lattice).
\end{enumerate}
The trivial enrichment is given by the identity $\pi=\operatorname{Id}\colon N\to N'=N$.
\end{definition}

\begin{definition} A {\em morphism of enriched fans} 
\[(\Delta_1,\pi_1\colon N_1\to N_1')\to (\Delta_2,\pi_2\colon N_2\to N_2')\] is a homomorphism $f\colon N_1\to N_2$ that fits into a 
commutative diagram 
\[\xymatrix{
N_1\ar[d]_{\pi_1}\ar[r]^f &N_2\ar[d]^{\pi_2}\\
N'_1\ar[r]^{f'}& N'_2 }\]
such that $f'$ induces a morphism of fans $\Delta_1\to\Delta_2$. 
\end{definition}

\begin{definition} A {\em partial compactification} of an enriched fan $(\Delta,\pi\colon N\to N')$ is a finite set of rays $R$ in $N_{\R}$ that are rational, i.e.,~spanned by vectors in $N$. We will denote this data by $(\Delta,\pi\colon N\to N',R)$.
For a cone $\sigma\in\Delta$, write $R_{\mathring{\sigma}}\coloneqq R\cap \pi^{-1}(\mathring{\sigma}).$ A morphism of partial compactifications
\[f\colon (\Delta_1,\pi_1\colon N_1\to N_1',R)\to (\Delta_2,\pi_2\colon N_2\to N_2',\varnothing)\] is a morphism of enriched fans 
\[(\Delta_1,\pi_1\colon N_1\to N_1')\to (\Delta_2,\pi_2\colon N_2\to N_2')\] 
with $f(R)=\{0\}$.
\end{definition}

\begin{example}
  The following enriched fans with partial compactification will play the role of base objects:
  \begin{enumerate}
      \item $\Delta_S\coloneqq (\{0\},\pi\colon \{0\}\to\{0\},\varnothing)$; and 
      \item $\Delta_{S^\dagger}\coloneqq (\{0\},\pi\colon\Z\to \{0\},\varnothing)$.
  \end{enumerate}
\end{example}

Any enriched fan with partial compactification has a trivial morphism to $\Delta_S$.

\begin{definition}
Let $(\Delta,\operatorname{Id}\colon N\to N,R)$ be a trivially enriched fan with partial compactification.
Let $\tau$ be a cone of $\Delta$. 
 Define the {\em induced enriched structure with partial compactification} on the star-quotient 
 \[(\Delta,\operatorname{Id}\colon N\to N,R)_\tau\coloneqq (\Delta_\tau,\pi\colon N\to N/N_\tau, R')\]
 by setting $\pi$ to be the projection
 and letting $R'$ be the subset of $R$ consisting of the vectors $\rho$ contained in some cone $\sigma$ containing $\tau$, i.e.,~$R'\coloneqq R\cap \bigcup_{\sigma\supseteq \tau} \sigma.$
The {\em induced enriched structure with partial compactification} on the open star $\Star^{\circ}_{\tau}(\Delta)$ is given by $(\Star_\tau^\circ(\Delta),\operatorname{Id}\colon N\to N,R')$
\end{definition}

While $\Star_\tau^\circ(\Delta)$ is not a polyhedral complex, the constructions that we will apply to it, in particular, construction of tropical differentials will still be valid.

\begin{definition}
    A {\em rational polyhedral complex with partial compactification} $(\Sigma,N,R)$ is the following data:
    \begin{enumerate}
      \item A finitely generated free abelian group $N$;
      \item a completable rational polyhedral complex $\Sigma$ in $N_{\R}$; and
      \item a partial compactification (i.e.,~a finite set of rational rays) $R$ in $N_{\R}$.
    \end{enumerate}
  We will often suppress the triple and use $\Sigma$ to refer to the above object.
\end{definition}

\begin{definition}
  Let $\Sigma$ be a completable polyhedral complex in $N_{\R}$. For a (poset) open subset $U$ of $\Sigma$, the {\em cone over $U$}, $\tilde{U}$ is the union of cones in $N_\R \times \R$ of the following types:
  \begin{enumerate}
    \item for each $P \in U$, the cone $\tP$ is the closure in $N_\R \times \R_{\geq 0}$ of the set
    \[\{(x,a) \subset N_\R \times \R_{>0} : \frac{x}{a} \in P\}; \text{ and}\]
    \item for each $P$ in $\Sigma$, the cone $P_0\coloneqq \tP\cap (N_\R\times\{0\})$.
    \end{enumerate}
    For a partial compactification  $(\Sigma,N,R)$, the induced enriched fan with partial compactification on $\tSigma$ is the following: the enriched structure is given by $\operatorname{Id}\colon N\times \Z\to N\times \Z$; and the partial compactification is given by the set $R\times\{0\}\subset N_\R\times\R$. 
    There is a natural morphism of enriched fans with partial compactifications
    \[f\colon (\tSigma,N\times \Z\to N\times \Z,R\times\{0\})\to \Delta_{S^\dagger}=(\{0\},\Z\to \{0\},\varnothing)\]
    where the map is given by projection onto the second factor $N\times\Z\to\Z$.
\end{definition}

For $U=\Sigma$, by the main result of \cite{BGS}, we can conclude that $\tSigma$ is a fan. In this case, the set of cones $\{P_0\}$ are the recession fan $\Sigma_0\subset N\times \{0\}\cong N$.
We say $\Sigma$ is {\em unimodular} if $\tSigma$ is unimodular and all vertices of $\Sigma$ belong to $N$.

\begin{definition} \label{d:starquotientcomplex} The {\em star-quotient of a polyhedron $P$ in $(\Sigma, N, R)$ over $\Delta_{S^\dagger}$} is the star-quotient of $\tilde{P}$ in $\tSigma$,
\[(\tilde{\Sigma}_{\tilde{P}},N\times \Z\to ((N\times\Z)/(N\times\Z)_{\tilde{P}}),R')\]
as per Definition~\ref{d:starquotientfan}. This comes equipped with a morphism to $\Delta_{S^\dagger}$.
Similarly, the {\em enriched structure with partial compactification on the star-quotient over $\Delta_S$} is
 $(\Sigma_P,N\to N/N_P,R')$.
 \end{definition}

A refinement of a rational polyhedral complex with partial compactification $(\Sigma,N,R)$ is one of the form $(\Sigma',N,R)$ where $\Sigma'$ is a refinement of $\Sigma$.

\subsection{Bergman fans}

\begin{definition}
    A {\em weighted balanced polyhedral complex} is a pure-dimensional polyhedral complex where each top-dimensional polyhedron is given an integer weight, and the weights satisfy the balancing condition (see, e.g.,~\cite[\S 3.4]{MaclaganSturmfels}). 
\end{definition}

The property of being a weighted balanced polyhedral complex is inherited by star-quotients of polyhedra.

The {\em Bergman fan} \cite{AK:Bergman} attached
to a loopless matroid $\M$ is a central example of a weighted balanced rational polyhedral fan in Euclidean space.
We may assume the ground set $ \left| \M \right| $ of $\M$ is the set $\{0,1,\dots,n\}$. 
To define the Bergman fan attached to $\M$, let $N=\Z^{n+1}/\Z\mathbf{1}$ where $\mathbf{1}$ is the diagonal vector. Let $e_0,\dots,e_n$ be the images in $N$ of the standard basis vectors of $\Z^{n+1}$.
The {\em Bergman fan} $\Sigma(\M) $ in $N_\R$ has rays $ \mathbb{R}_{\geq 0} e_{I} $  for flats $I$ of $\M$, where 
$e_{I} = \sum_{i \in I } e_{i}$.
The $k$-dimensional cones of $ \Sigma(\M) $ are given by $\Span_{\geq 0}(e_{I_{1}}, \dots, e_{I_{k}})$
where
$I_{1} \subsetneq \dots \subsetneq I_{k} $ is a flag of flats.
Bergman fans satisfy the balancing condition when their top-dimensional cones are equipped with multiplicity $1$.

\begin{definition} \label{d:smoothcomplex}
A balanced weighted rational polyhedral complex $\Sigma$ is {\em smooth} if all top-dimensional polyhedra have weight $1$ and the star-quotient of every face $F$ has the same support set as the Bergman fan of some matroid with respect to some choice of basis of $N/N_F$. 
\end{definition}

See \cite{KS:TMNF} for more details. Smoothness is a property of the support set of the polyhedral complex provided every top-dimensional face has multiplicity $1$:

\begin{lemma} \label{l:anystructure}
  Let $\Sigma$ be a balanced weighted polyhedral complex such that every top-dimensional polyhedron has weight $1$ and contains no non-trivial linear subspace. Then $\Sigma$ is smooth if and only if for all $x\in |\Sigma|$, $T_x\Sigma$ has support equal to that of a Bergman fan with respect to some basis of $N$.
\end{lemma}

\begin{proof}
  Observe that $\Sigma$ is smooth if and only if for every vertex $v\in\Sigma_{(0)}$, $T_v\Sigma$ is the support  set of $\Sigma(\M)$, the Bergman fan of a matroid $M$. This is a consequence of the support of the star quotient of a Bergman fan being the support of a Bergman fan.
  Since every tangent cone of the support of a Bergman fan is the support of a  Bergman fan, the conclusion follows.
\end{proof}

\subsection{Tropical differential forms}

We will define tropical differential forms (see, e.g.,~\cite{grossShokrieh,IKMZ}) for enriched fans with partial compactifications.  

\begin{definition}
Let
$f\colon (\Delta_1,\pi_1\colon N_1 \to N_1',R)\to (\Delta_2,\pi_2\colon N_2\to N'_2,\varnothing)$
be a morphism of enriched fans with partial compactifications. Give $\Delta_1$ the poset topology. The sheaf of {\em complex-valued tropical relative differential forms attached to $f$} is induced by the presheaf on $\Delta_1$,
\begin{align*}
    \Omega^1_{\Delta_1/\Delta_2}(U)&\coloneqq \left(\frac{\bigcap_{\sigma:\sigma^\circ\subseteq U} \Span_{\R}(R_{\mathring{\sigma}})^{\perp_{M_{1,\R}}}
+\Span_{\R}(U)^{\perp_{M'_{1,\R}}}}{f^*M_{2,\R}+\Span_{\R}(U)^{\perp_{M_{1,\R}'}}}\right)_{\C}\\
&\subseteq \left(\frac{M_{1,\R}}{f^*M_{2,\R}+\Span_{\R}(U)^{\perp_{M_{1',\R}}}}\right)_{\C}.
\end{align*}
where the perpendiculars are taken in $M_{1,\R}$ and $M'_{1,\R}$, as noted. For $U_2\subseteq U_1$, the restriction map is given by
taking the further quotient. We will write $\Omega^1_{\Delta_1/\Delta_2}$ for the induced sheaf as well.
\end{definition}

Here, we view tropical differential forms as restrictions of elements of $M_{1,\R}$, i.e.,~elements of the dual space to $N_{1,\R}$. They differ from usual tropical differentials in two ways:
\begin{enumerate}
    \item we allow them to be non-trivial on hidden directions, i.e.,~those in $\ker(\pi\colon N\to N')$,
    \item we force them to vanish on the rays in $R$.
\end{enumerate}
We consider elements of $M_{1,\R}/\Span_\R(U)^{\perp_{M'_{1,\R}}}$ to be differential forms on $U$ (that are allowed to be non-trivial on hidden directions). All of our examples will be over $\Delta_S$ (where $M_2=\{0\}$) or $\Delta_{S^\dagger}$ (where $M_2=\Z$). The examples over $\Delta_{S^\dagger}$ will occur as cones over polyhedral complexes.

\begin{definition} \label{d:tropdiffpoly}
    For a rational polyhedral complex with partial compactification, $(\Sigma,N,R)$, we define a sheaf of tropical differentials $\Omega^1_{\Sigma}$ on the poset topology of $\Sigma$. Let $\tilde{\Sigma}\subseteq N_{\R}\times\R$ be the cone over $\Sigma$ equipped with the morphism of enriched fans with partial compactifications
    \[f\colon (\tSigma,N\times \Z\to N\times \Z,R\times\{0\})\to \Delta_{S^\dagger}=(\{0\},\Z\to \{0\},\varnothing).\]
    For an open $U$ in the poset topology, we set 
    $\Omega^1_{\Sigma}(U)\coloneqq \Omega^1_{\tSigma/\Delta_{S^\dagger}}(\tilde{U})$
    where $\tilde{U}$ is considered to be an open set in $\tilde{\Sigma}$. The sheaf $\Omega^1_{\Sigma}$ is the sheafification of the above.
\end{definition}

By unpacking the above definition, we see that for $U$, the open star of a polyhedron,
\[\Omega^1_{\Sigma}(U)=\bigcap_{P\subseteq U} \Span(R_{\mathring{P}})^{\perp_{N_{U,\R}^\vee}}\subseteq N_{U,\R}^\vee\]
where $N_{U,\R}=\sum_{\mathring{P}\subset U} N_{P,\R}.$
In other words, $\Omega^1_{\Sigma}(U)$ is the vector space of tropical differentials on $U$ that vanish on the vectors in $R$ that are contained in the recession cone of a polyhedron in $U$.

\begin{definition} \label{d:tropidifffan}
    For an enriched fan with partial compactification, $(\Delta,\pi\colon N\to N',R)$, the {\em sheaf of tropical differentials} is defined to be $\Omega^1_{\Delta/\Delta_S}$.
\end{definition}

\begin{rem} \label{r:triviallyenriched1forms}
    For a trivially enriched fan $(\Delta,N\to N,R)$, there are two candidate sheaves of tropical differentials: the one from considering it as a polyhedral complex (Definition~\ref{d:tropdiffpoly}); or the one from considering it as a fan (Definition~\ref{d:tropidifffan}). These are easily seen to be isomorphic.
    Moreover, for a polyhedron $P$ in $(\Sigma,N,R)$, the sheaves of tropical differentials on $\Sigma_P$ attached to the star quotients $\Omega^1_{\tilde{\Sigma}_{\tilde{P}}/\Delta_{S^\dagger}}$ and $\Omega^1_{\Sigma_P/\Delta_S}$ are canonically isomorphic.
\end{rem}


\begin{definition}
Let $\Sigma_{\max}$ be the (poset) open subset of $\Sigma$ consisting of points $x\in\Sigma$ whose tangent cone $ T_{x}\Sigma $ is a linear subspace, and write $\iota\colon\Sigma_{\max}\to \Sigma$ for the inclusion.
The {\em sheaf of tropical $p$-forms}, denoted $ \Omega_{\Sigma}^{p}$ is the image of the morphism
\[ \bigwedge\nolimits^p\Omega_{\Sigma}^{1} \to \iota_{*}\left(\bigwedge\nolimits^p\Omega_{\Sigma}^{1}
\big|_{\Sigma_{\max}}\right). \]
\end{definition}

We define tropical differential forms on fans analogously.
In other words, we allow tropical differential forms to be determined by their restriction to $\Sigma_{\max}$.
As per \cite{grossShokrieh}, there is no ambiguity in the notation for denoting $\Omega^1_{\Sigma}$.

\begin{figure}
    \begin{center}
    \begin{tikzpicture}
     \filldraw (-2,0) circle (2pt) ;
     \filldraw (2,0) circle (2pt) node[anchor = south east]{$v$} ;
     \draw (-2,0) -- (2,0);
     \draw[->] (-2,0) -- (-3.4, -1.4);
     \draw[->] (-2,0) -- (-2, 2.2);
     \draw[->] (2,0) -- (2,2.2) node[anchor = south]{$e_{2}$};
     \draw[->] (2,0) -- (3.4,-1.4) node[anchor = west]{$e_{1}-e_{2} $};
     \fill[color = gray,semitransparent] (-2,0) -- (-2,2) -- (2,2)-- (2,0) -- (-2,0);
     \fill[color = gray,semitransparent] (2,0) -- (2,2) -- (3.2,-1.2)-- (2,0);
     \fill[color = gray,semitransparent] (-2,0) -- (-2,2) -- (-3.2,-1.2)-- (-2,0);
     \fill[color = gray,semitransparent] (-2,0) -- (2,0) -- (3.2,-1.2)-- (-3.2,-1.2) -- (-2,0);
    \end{tikzpicture}
    \caption{A Rational Polyhedral Complex in $\mathbb{R}^{2} $}
    \label{f:exampleComplex}
    \end{center}
\end{figure}
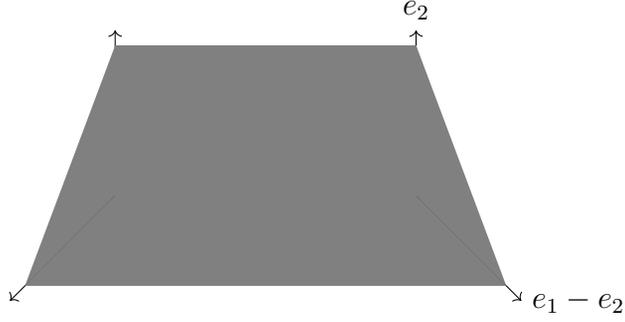

\begin{example}
    Let $ \Sigma \subseteq \mathbb{R}^{2} $ be the rational polyhedral complex in Figure \ref{f:exampleComplex} with partial compactification $ (\Sigma, \mathbb{Z}^{2}, \{\mathbb{R}_{\geq 0} \cdot e_{1} \}) $. Note that in the complex, all $ 2 $-cells are unbounded.

    We will compute the sheaves of tropical differentials on the 
    open set $ \Star^{\circ}_{v}(\Sigma) $. The cone over $ \Sigma $
    is a subset of $ \mathbb{R}^{2} \times \mathbb{R} $; call the
    unit vector in the last factor $ e_{3} $.
    Then \[
        \Omega^{1}_{\Sigma}(\Star^{\circ}_{v}(\Sigma)) \coloneqq
        \Omega^{1}_{\widetilde{\Sigma}/ \Delta_{S^{\dagger}}}\left(
        \widetilde{\Star_{v}^{\circ}(\Sigma)} \right)
        = \frac{\mathbb{R} e_{2}^{*}  + \mathbb{R} e_{3}^{*}}{\mathbb{R} e_{3}^{*}}
        \cong \mathbb{R} e_{2}^{*}
    \] 
    As this is of real dimension $ 1$, $ \Omega_{\Sigma}^{i}\left( 
    \Star^{\circ}_{v}(\Sigma) \right) = 0 $  for $ i > 1 $.
\end{example}

We can pull back differential forms to star-quotients of faces. The following is proven by unwinding definitions.

\begin{lemma}\label{l:starAndStarQuotientDifferentials}
  Let $(\Sigma,N,R)$ be a rational polyhedral complex with partial compactification. Let  $(\Sigma_P,N\to N/N_P,R')$ be the enriched structure on the star-quotient fan over $\Delta_S$. 
  Let $U$ be any open subset of $\Sigma$ containing the open star of $P$ in $\Sigma$. Then 
  there is a natural homomorphism
  \[\Omega^1_{\Sigma}(U)\to \Omega^1_{\Sigma_P}(\Sigma_P)\]
  which is an isomorphism when $U$ is equal to the open star of $P$.
\end{lemma}

In the case of trivial partial compactification, sheaves of tropical differential 
forms may be defined on the Euclidean topology and hence depend only on the support of the polyhedral complex (see \cite{grossShokrieh}). 
For non-trivial partial compactifications, proof 
is required to show tropical differentials are unchanged under refinement.

\begin{lemma}\label{l:tropDiffRefinement}
    Let $ (\Sigma, N, R) $ be a polyhedral complex with partial compactification, and let $ \Sigma' $ be a refinement of $ \Sigma $. For a face
    $ F \in \Sigma $, we have 
    \[ 
        \Omega^{i}_{\Sigma}\left( \Star^{\circ}_{F}(\Sigma) \right)
        \cong \Omega^{i}_{\Sigma'}\left( \Star^{\circ}_{F}(\Sigma) \right)
    \]
    for $ i \geq 0 $.
\end{lemma}

\begin{proof}
    It suffices to show this for $i=1 $. We may suppose that the affine span of $\Sigma$ is $N_\R$.
    Let $ \Omega^{1}_{\Sigma',\varnothing} $ be the sheaf of tropical 
    differential forms on $ (\Sigma', N, \varnothing) $. 
    Notice that, viewed as subspaces of $ \Omega^{1}_{\Sigma', \varnothing}
    \left( \Star_{F}^{\circ}(\Sigma)\right) $, we have
    \[ 
         \Omega^{1}_{\Sigma}\left( \Star_{F}^{\circ}(\Sigma)\right)
         \subseteq \Omega^{1}_{\Sigma'} \left( \Star_{F}^{\circ}
         (\Sigma) \right) 
    \]
    To show the containment
    \[ 
    \Omega^{1}_{\Sigma}\left( \Star_{F}^{\circ}(\Sigma)\right)
        \supseteq \Omega^{1}_{\Sigma'} \left( \Star_{F}^{\circ}
         (\Sigma) \right),
    \]
    let $ a \in \Omega^{1}_{\Sigma', \varnothing}
    \left( \Star_{F}^{\circ}(\Sigma)\right) \setminus \Omega^{1}_{\Sigma}\left( \Star_{F}^{\circ}(\Sigma)\right).$
    Unwinding definitions, there is some $P \in \Sigma $ with $ \mathring{P} \subseteq \Star_{
    \widetilde{F}}(\widetilde{\Sigma}) $ and some $ r \in 
    R_{\mathring{P}_0} $ such that $ \langle a, r \rangle \neq 0 $.
    As $r$ is in the relative interior of the recession cone of $P$ with $\mathring{P}\subseteq \Star_{F}^{\circ}(\widetilde{\Sigma}) $, there is some $ Q \in \Sigma' $ with $\mathring{Q}\subseteq  \Star_{F}^{\circ}(\widetilde{\Sigma})$ such that $r\in \mathring{Q}_0$.
    Therefore $a$ is not an element of 
    $\Omega_{\Sigma'}^{1}\left(\Star_{F}^{\circ}(
    \Sigma)\right) $.
\end{proof}

\section{The Unipotent Tropical Fundamental Group} \label{s:tropicalpi1}

\subsection{Definitions}
We now define the unipotent tropical fundamental group as the Tannakian fundamental group of tropical vector bundles equipped with integrable connection.

\begin{definition}
Let $\Sigma$ be a connected rational polyhedral complex in $N_\R$. A rank $n$ {\em tropical vector bundle} on $\Sigma$ is a locally constant sheaf $E$ (in the poset topology) of $n$-dimensional $\C$-vector spaces.
Given a rational polyhedral complex with partial compactification $(\Sigma,N,R)$, a {\em connection} on $E$ is a global section
\[ 
    \theta \in \Gamma\left(\Sigma, \Omega_{\Sigma}^{1} \otimes \End(E) \right),
\] 
and $ \theta $ is said to be {\em integrable} if $ \theta \wedge \theta = 0 $.

Given two tropical vector bundles with integrable connection $(E_1,\theta_1)$, $(E_2,\theta_2)$, a homomorphism from $E_1$ to $E_2$ is a morphism of sheaves $T\colon E_1\to E_2$ such that for all local sections $s$ of $E_1$, $\theta_2 T(s)=T(\theta_1 s)$ in $E_2\otimes \Omega^1_{\Sigma}$.

For an enriched fan with partial compactification $(\Delta,N\to N',R)$, tropical vector bundles with integrable connections are defined analogously by using the sheaf of tropical differentials $\Omega^1_{(\Delta,N\to N',R)/\Delta_S}$.
\end{definition}

By Remark~\ref{r:triviallyenriched1forms}, there is no ambiguity about whether to consider a fan as a polyhedral complex.
A tropical vector bundle on a fan is always a constant sheaf because the fan itself is the minimal open set containing the origin.
The unit tropical vector bundle with integrable connection $\mathbf{1}$ is the trivial bundle with fiber $\C$ and connection form $\theta=0$. 

Let $\sC^{\trop}(\Sigma,N,R)$ (abbreviated $\sC(\Sigma)$)  denote the abelian tensor category of tropical vector bundles with integrable connection 
 under tensor product and duality as follows:
\[(E_1,\theta_1)\otimes (E_2,\theta_2)=(E_1\otimes E_2,\theta_1\otimes 1+1\otimes \theta_2),\  (E_1,\theta_1)^\vee=(E_1^\vee,-\theta_1^\vee).\]
The space of {\em horizontal sections of $(E,\theta)$ on $\Sigma$} is
\[\Gamma(\Sigma,(E,\theta))=\{s\in\Gamma(\Sigma,E)\mid \theta s=0\}.\]
One can see that
\begin{equation} \label{e:tropicalhoms}
\Hom\left((E_1,\theta_1),(E_2,\theta_2)\right)=\Gamma(\Sigma,(E_1,\theta_1)^\vee\otimes (E_2,\theta_2)).    
\end{equation}
A {\em unipotent tropical vector bundle with connection} is a unipotent object in that category. 
To a polyhedron $P$ of $\Sigma$, we define a fiber functor $\omega_P\colon \sC(\Sigma)\to \Vect_{\C}$ by
\[\omega_P(E)\coloneqq E(\Star^\circ_P(\Sigma)),\]
yielding the following:

\begin{prop} The category $\sC^{\trop,\un}(\Sigma,N,R)$ is a neutral Tannakian category with fiber functor $\omega_P$, and analogously for $\sC^{\trop,\un}(\Delta,N\to N',R)$ for enriched fans with partial compactifications over $\Delta_S$. 
\end{prop}

Because the above category is defined in reference to sheaves in the poset topology, we may extend the definition to (poset) open subspaces of connected rational polyhedral complexes.

The above construction of the category of unipotent tropical vector bundles with integrable connection holds for enriched fans with partial compactification relative to a base fan by using the sheaf of relative tropical differential forms. Indeed, if we have a morphism 
\[f\colon (\Delta_1,\pi_1\colon N_1\to N'_1,R)\to (\Delta_2,\pi_2\colon N_2\to N'_2,\varnothing),\]
we define tropical vector bundles on $\Delta_1$ (which all must be trivial) together with connections by making use of $\Omega^1_{\Delta_1/\Delta_2}$ in place of $\Omega^1_{\Sigma}$.

\begin{definition}
  The {\em unipotent tropical fundamental group} $\pi_1^{\un}(\Sigma,\omega_P)$ is the Tannakian fundamental group of $\sC^{\un}(\Sigma)$ with respect to the fiber functor $\omega_P$ and analogously in the fan case.
\end{definition}

\begin{lemma} \label{l:starquotientequivalence}
  Let $(\Sigma,N,R)$ be a rational polyhedral complex with partial compactification. 
  For a polyhedron  $P\in\Sigma$, let  $\Star_{\tilde{P}}^\circ(\tilde{\Sigma})$ denote the induced enriched structure with partial compactification on the open star, and let  $\Sigma_P$ denote the enriched structure of the star-quotient over $\Delta_S$ (as in Definition~\ref{d:starquotientcomplex}). Then there is a natural equivalence of categories between $\sC^{\un}(\Star_P^\circ(\Sigma))$ and $\sC^{\un}(\Sigma_P)$.
  Moreover for an inclusion of polyhedra $P_1\subset P_2$, the inclusion
      $\Star^\circ_{\tilde{P}_2}(\tilde{\Sigma})\hookrightarrow\Star^\circ_{\tilde{P}_1}(\tilde{\Sigma})$
  induces a pullback functor
  \[\sC^{\un}(\Sigma_{P_1})\cong \sC^{\un}(\Star^\circ_{P_1}(\Sigma))\to
   \sC^{\un}(\Star^\circ_{P_2}(\Sigma))\cong\sC^{\un}(\Sigma_{P_2}).\]
\end{lemma}

This lemma is a consequence of Lemma~\ref{l:starAndStarQuotientDifferentials} which states that the open star and star-quotient have isomorphic sheaves of tropical differential forms. Because on either space, the unipotent locally constant sheaves of $\C$-vector spaces are trivial, the categories of tropical vector bundles with integrable connections are equivalent. This follows from the observation that any unipotent vector bundle on $ \Sigma$ is trivialized on the support of an open star $U$ in $ \Sigma $. Indeed, this is proven by induction on the length of the tropical vector bundle as a unipotent object, as extensions on $U$ are classified by direct sums of $ H^{1}(\left| U
\right|, \mathbb{C}) =0 $.

\begin{rem} \label{r:refinement} Given a refinement $\Sigma'$ of $\Sigma$, the pullback functor $\sC^{\un}(\Sigma)\to \sC^{\un}(\Sigma')$ is an equivalence of categories. 
Indeed, the tropical differentials on the support of open stars in $ \Sigma $ are unchanged by the refinement (Lemma \ref{l:tropDiffRefinement}) and the categories of unipotent locally constant sheaves of $\C$-vector spaces on  $ \Sigma' $ and $ \Sigma $ are equivalent. \end{rem}

\subsection{Descent for the unipotent tropical category}

By gluing tropical vector bundles with integrable connections, we will give a descent statement and hence an equivalence of categories.
Let $(\Sigma,N,R)$ be a connected  rational polyhedral complex with partial compactification, and suppose that no face of $\Sigma$ contains a line.
 Let $\Gamma$ be the $2$-skeleton of the union of the bounded faces of $\Sigma$. Suppose that $\Gamma$ is a simplicial complex which is true if $\Sigma$ is unimodular. Define a Tannakian category $\sC_\Gamma$ over $\Gamma$ by setting $\sC_F\coloneqq\sC^{\un}(\Sigma_F)$ with transition morphisms  for $F_1\subset F_2$ $\sC_{F_1}\to \sC_{F_2}$ given by Lemma~\ref{l:starquotientequivalence}. 
The following is a matter of unpacking definitions together with the observation that the union of the open stars of the bounded faces of $\Sigma$ cover $\Sigma$:

\begin{prop} \label{p:tropicaldescent}
    There is an equivalence of categories between $\sC^{\un}(\Sigma,N,R)$ and the unipotent completion of the descent category, $\sC(\Gamma)^{\un}$.
\end{prop}

Observe that objects of $\sC(\Gamma)$ are {\em locally unipotent}, that is, they only restrict to unipotent ones on each face, so one must pass to the unipotent subcategory. We can apply the above proposition to obtain a presentation for the fundamental group of $(\Sigma,N,R)$. 
Then, $\pi_1^{\un}(\Sigma,\omega_P)$ is isomorphic the unipotent completion of the group described in \cite[Section~3.2]{Stix:SvK} that is built out of the groups $\pi_1^{\un}(\Sigma_v,\omega_v)$ for vertices $v$ of $\Gamma$.

\subsection{The tropical fundamental group of a fan}

We describe the tropical fundamental group of a fan using the bar construction (see, e.g.,~\cite[Chapter~8.3.3]{Peters-Steenbrink}). Let $(\Delta,\pi\colon N\to N',R)$ be an enriched fan with partial compactification. Let $0$ be the origin in $N'$, considered as a cone in $\Delta$, and let $\omega_0$ be its attached fiber functor. 
Let $\Omega^\bullet(\Delta)$ be the tropical de Rham complex on $\Delta$, that is, the differential graded algebra with zero differential whose $p$th graded piece is the space of sections, $\Omega^p_{(\Delta,\pi\colon N\to N',R)}(\Delta)$. Let $I\Omega\coloneqq \Omega^{\geq 1}(\Delta)$ be its augmentation ideal. Define 
\[B\coloneqq \bigoplus_{s=0}^\infty (I\Omega)^{\otimes s}\]
to be the free tensor algebra on $I\Omega$ where we write $\eta_1\otimes\cdots \otimes\eta_s$ as $[\eta_1|\cdots|\eta_s]$. Give $B$ a grading where
\[\deg([\eta_1|\cdots|\eta_s])=\deg(\eta_1)+\cdots+\deg(\eta_s)-s,\]
so the degree $0$ elements are of the form $[\eta_1|\cdots|\eta_s]$ for $\eta_i\in\Omega^1$.
Let $d$ be the ``combinatorial'' differential on $B$ given by
\[d([\eta_1|\cdots|\eta_s])=\sum_{i=1}^s (-1)^{i+1}[J\eta_1|\cdots| J\eta_{i-1}| J\eta_i\wedge\eta_{i+1}|\eta_{i+2}|\cdots|\eta_s]\]
where $J\eta_i=(-1)^{\deg(\eta_i)}\eta_i.$
Equip $B$ with the coproduct
\[\Delta_B\colon B\to B\otimes B\]
given by 
\[[\eta_1|\cdots|\eta_s]\mapsto \sum_{i=0}^s [\eta_1|\cdots|\eta_i]\otimes
[\eta_{i+1}|\cdots|\eta_s]\]
and multiplication given by the shuffle product:
\[[\eta_1|\cdots|\eta_r]\otimes
[\eta_{r+1}|\cdots|\eta_{r+s}]=\sum_\sigma \operatorname{sgn}(\sigma)[\eta_{\sigma(1)}|\cdots|\eta_{\sigma(r+s)}]\]
where the sum is over shuffles $\sigma$ of type $(r,s)$, i.e.,~permuations of $\{1,\dots,r+s\}$ such that $\sigma^{-1}(1)<\dots<\sigma^{-1}(r)$ and $\sigma^{-1}(r+1)<\dots<\sigma^{-1}(r+s)$, and the sign is given by giving $\eta_i$ weight $\deg(\eta_i)-1$.
Now $H^0(B)=\ker(d\colon B^0\to B^1)$ is generated by elements of the form $[\eta_1|\cdots|\eta_s]$ where for $1\leq i\leq s-1$, $\eta_i\wedge \eta_{i+1}=0$. 

Analogues of the following theorem have appeared in the literature and owe much to the work of Chen \cite{Chen:bulletin}, but we include a proof here for completeness. See \cite{Terasoma} for an extension. Below, the Hopf algebra structure on $H^0(B)$ makes $\Spec H^0(B)$ into a group scheme.

\begin{theorem}
    Let $(\Delta,N\to N',R)$ be an enriched fan with partial compactification. There is an isomorphism
    \[\pi_1^{\un}((\Delta,N,R),\omega_0)\cong\Spec H^0(B).\]
\end{theorem}

\begin{proof}
  By the equivalence between modules and comodules (see, e.g.,~\cite[Chapter~2]{Jantzen:representations}), it suffices to show that there is an equivalence of categories between $H^0(B)$-comodules and unipotent tropical vector bundles on $\Delta$. Write $\Omega^1$ for $\Omega^1(\Delta)$.

  By applying $\omega_0$, we may treat a unipotent tropical vector bundle $(E,\eta)$ as a vector space $E_{0}$ equipped with an element $\theta\in\Omega^1\otimes \End(E_{0})$ satisfying  
  $\theta\wedge\theta=0$, and such that there exists $n\in\Z_{\geq 1}$ such that $\theta^n$ maps to $0$ under the map
  \[
  (\Omega^1\otimes\End(E_0))^{\otimes n}\cong (\Omega^1)^{\otimes n}\otimes\End(E_0)^{\otimes n}\to (\Omega^1)^{\otimes n}\otimes\End(E_0)
  \]
  given by multiplication of elements of $\End(E_0)$.
  We give $E_{0}$ a comodule structure $\Delta_\theta\colon E_{0}\to E_{0}\otimes H^0(B)$ by
  \[v\mapsto v+ c_1(v \otimes [\theta]) + c_2(v \otimes [\theta|\theta])+\dots\]
  where the contraction $c_k\colon E_0\otimes (\Omega^1\otimes\End(E_{0}))^{\otimes k}\to E_{0}\otimes (\Omega^1)^{\otimes k}$ is given by 
  \[v\otimes A_1\otimes A_2\otimes\dots\otimes A_k\mapsto vA_1A_2\dots A_k.\]
  Note that the above sum is finite because $\theta$ is nilpotent.

  Given an $H^0(B)$-comodule $E_0$, treat the comodule structure as a map 
  \[E_0\to E_0\otimes H^{0}(B),\] and thus as an element $\Delta_\theta\in \End(E_0)\otimes H^{0}(B)$. Pick a basis $\eta_1,\dots,\eta_m$ of $\Omega^1$, and write
  \[\Delta_{\theta}=\sum_s \sum_{(i_1,\dots,i_s)} A_{i_1 \dots i_s}\otimes [\eta_{i_1}| \cdots| \eta_{i_s}]\]
  for $A_{i_1 \dots i_s}\in \End(E_0)$ where all but finitely many summands are $0$. The comodule condition (i.e.,~the intertwining of $\Delta_B$ and $\Delta_\theta$) gives
  \[\begin{split}
    \sum_{s,t}\sum_{(i_1,\dots,i_s)}\sum_{(j_1,\dots,j_t)}A_{i_1 \dots i_s}A_{j_1 \dots j_t} \otimes  [\eta_{i_1}| \cdots| \eta_{i_s}] \otimes  [\eta_{j_1}| \cdots| \eta_{j_t}]\\
    =\sum_{s,t}\sum_{(i_1,\dots,i_s)}\sum_{(j_1,\dots,j_t)} A_{i_1 \dots i_s j_1 \dots j_t}\otimes [\eta_{i_1}| \cdots| \eta_{i_s}] \otimes  [\eta_{j_1}| \cdots| \eta_{j_t},]
  \end{split}\]
   which implies, by induction, 
  \[\Delta_{\theta}=I+[\theta]+[\theta|\theta]+[\theta|\theta|\theta]+\dots\]
  for $\theta=\sum A_i\otimes \eta_i$. Since this sum is finite, $\theta$ corresponds to a nilpotent connection.
  Now, $\Delta_\theta\in \End(E_0)\otimes H^0(B)$ implies $\theta\wedge\theta=0$. Hence $\theta$ corresponds to an integrable unipotent connection on the tropical vector corresponding to $E_0$. 

  It is straightforward that the above correspondence induces a fully faithful functor.
\end{proof}

By specializing this construction to Bergman fans and identifying tropical differential forms with the Orlik--Solomon algebra \cite{Zharkov:Orlik}, one obtains a description of the tropical unipotent fundamental group analogous to Kohno's description of the unipotent fundamental group of a hyperplane arrangement complement \cite{Kohno:holonomy,Kohno:bar}.

\subsection{Examples}

For a generic tropical line $L$ in $N_\R=\R^2$ with $R=\varnothing$ , $\Omega^1$ is generated by the tropical differentials $\eta_1=dx,\eta_2=dy$. Since $\eta_1\wedge \eta_2=0$, $H^0(B)=\C\langle x,y\rangle$, the polynomial ring in two non-commuting variables under the shuffle product and the coproduct dual to concatenation. Hence, $\pi_1^{\un}(L,\omega_0)$ is the free pro-unipotent group on two generators. 

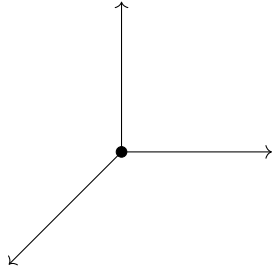
\begin{figure}
\centering
\begin{tikzpicture}
\filldraw[black] (0,0) circle (2pt);
\draw[->] (0,0) --(2,0);
\draw[->] (0,0) -- (0,2);
\draw[->] (0,0) -- (-1.5,-1.5);
\end{tikzpicture}
\caption{The Tropical Line}\label{tropicalLine}
\end{figure}

Now, consider a tropical elliptic curve $\Sigma$ in the plane $N_\R=\R^2$. That is, $\Sigma$ is a tropical hypersurface with $b_1(\Sigma)=1$, so that it corresponds to a cycle with some trees attached at vertices
(see Figure \ref{fig:ellipticCurve}).

\begin{lemma}
    For a tropical elliptic curve $(\Sigma,N,R)$ in $N_{\R}=\R^2$ where $R$ contains the directions of all rays of $\Sigma$, the category $\sC^{\un}(\Sigma)$ is equivalent to finite-dimensional vector spaces equipped with a pair of commuting unipotent automorphisms. Consequently, for $v$, a vertex of $\Sigma$, $\pi_1^{\un}(\Sigma,\omega_v)\cong\operatorname{G}_a^2$.
\end{lemma}

\begin{proof}
For now, pick a vertex $v$ in the cycle of $\Sigma$. By applying $\omega_v$ to a tropical vector bundle $E$ on $\Sigma$, we obtain a finite-dimensional vector space $V$. The tropical vector bundle, as a locally constant sheaf, is determined by its monodromy along the cycle of $\Sigma$, and thus its data is equivalent to a unipotent automorphism $T\colon V\to V$.
Pick an edge $e$ at $v$ contained in the cycle and a local coordinate $t$ along $e$.
The space of tropical $1$-forms on $\Sigma$ is $1$-dimensional, and let $\eta$ be the unique $1$-form inducing $dt$ on $e$. Then, the tropical connection $\eta$ restricts to $e$ as $S \otimes dt$ for $S$, a nilpotent endormorphism of $V$. The condition that $S$ comes from an element of $\Gamma(\Omega^1\otimes \End(E))$ is equivalent to $ST=TS$. Thus, we see that the data of $E$ is determined by $V$ and a pair of commuting unipotent matrices, $I + S$ and $T$. Hence, $\pi^{\un}_1(\Sigma,\omega_v)\cong \Ga^2$, the pro-unipotent completion of $\Z^2$.
For an arbitrary $v$, we note that we have an isomorphism between fundamental groups at different base-points.
\end{proof}

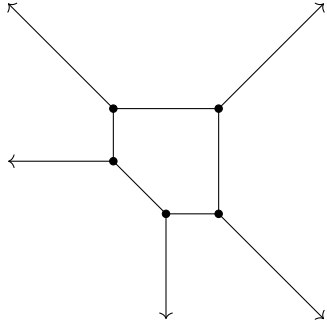
\begin{figure}
    \begin{center}
    \begin{tikzpicture}[scale=0.7]
        \filldraw[black] (-1,1) circle (2pt);
        \filldraw[black] (1,1) circle (2pt);
        \filldraw[black] (-1,0) circle (2pt);
        \filldraw[black] (0,-1) circle (2pt);
        \filldraw[black] (1,-1) circle (2pt);
        \draw[-] (-1,1) --(1,1) -- (1,-1) -- (0,-1) -- (-1,0) -- (-1,1);
        \draw[->] (-1,1) -- (-3,3);
        \draw[->] (-1,0) -- (-3,0);
        \draw[->] (0,-1) -- (0,-3);
        \draw[->] (1,-1) -- (3,-3);
        \draw[->] (1,1) -- (3,3);
    \end{tikzpicture}
    \caption{A Tropical Elliptic Curve.}
    \label{fig:ellipticCurve}
    \end{center}
\end{figure}

\section{Log Structures, Toric Varieties, and subschemes} \label{s:logstructures}

In this section, we will discuss log structures induced on subschemes of toric varieties and subschemes. Here, we will use some slightly nonstandard log structures induced by enriched fans with partial compactifications. This will give us the flexibility to work with the inverse image log structure on closed strata of toric schemes. The log structures will induce a sheaf of log differentials which can be recovered directly from combinatorial data. It is this technical observation that will make our correspondence theorem possible.

\subsection{Log structures}

We recall log schemes and recommend \cite{Ogus:logbook,Kato:Logstructures} for more details. For our purposes, it is enough to
work on the Zariski site rather than the \'{e}tale site.

\begin{definition}
    A {\em log scheme} $(X,M)$ is a scheme $X$ together with a sheaf of monoids $M$ on $X$ and a morphism of sheaves of monoids $\alpha\colon M\to\cO_X$ such that 
    $\alpha^{-1}(\cO_X^*)\to \cO^*_X$ is an isomorphism. A {\em morphism of log schemes} $(X,M_X)\to (Y,M_Y)$ is a morphism $f\colon X\to Y$ together with a homomorphism $f^\flat\colon M_Y\to f_*M_X$ that intertwines $\alpha_Y$ and $f_*\alpha_X$.
\end{definition}

The {\em trivial log structure} on $X$ is $(X,\cO_X^*)$ where $\alpha\colon \cO_X^*\to \cO_X$ is the usual inclusion.
We will typically specify a {\em pre-log structure} (i.e. a morphism of sheaves of monoids $\alpha\colon P\to \cO_X$) to induce the desired log structure $M \coloneqq (P\oplus \cO_X^*)/K$ where
$K$ is the subsheaf of monoids defined by $K\coloneqq \{(x,\alpha(x)^{-1})\mid x\in\alpha^{-1}(\cO_X^*)\}.$

\begin{definition} \cite[Definition~1.1.5]{Ogus:logbook}
    Given a log scheme $(X,M)$ and a morphism of schemes $f\colon X'\to X$, the {\em inverse image log structure} 
    \[f^*\alpha\colon f^*M\to \cO_{X'}\] 
    is the log structure structure attached to the pre-log structure
    \[f^{-1}(M)\to f^{-1}(\cO_X)\to \cO_{X'}.\] 
\end{definition}
For a closed immersion $i\colon Z\hookrightarrow X$ where $X$ underlies a log scheme $(X,M)$, we will write $M|_Z$ for $i^*M$. 

\begin{example}
    The {\em trivial point} is $S=\Spec \C$ with $M=\C^*\to \C=\cO_S$. The {\em standard log point} is $S^\dagger$ given by $\Spec \C$ equipped with the log structure  induced by $\N=\langle e\rangle\to \cO_S$ sending $e\mapsto 0$ (where $\N$ denotes the nonnegative integers). There is a natural morphism $S\to S^\dagger$ given by the identity on the underlying scheme and induced by $\N\to 0$ on the monoid. This morphism is analogous to the inclusion of a closed point into $\Spec \C\ps{t}$. On the other hand, there is a morphism $S^\dagger\to S$ (given by the usual inclusion of monoids) analogous to the structure morphism $\Spec \C\ps{t}\to\Spec \C$. 

    We can generalize the above to arbitrary DVRs. If $\cO$ is a DVR, we  equip $\Spec \cO$ with the log structure induced by $\N\to \cO$ sending $e\mapsto \pi$ for a uniformizer $\pi$. The log structure is independent of the choice of uniformizer.
\end{example}


We discuss {\em strict normal crossings divisors and schemes} following the discussion in \cite[Section~1.8]{Ogus:logbook}. 
 For a regular scheme $X$, a {\em divisor with strict normal crossings} is an effective divisor in $X$ that has the property that the scheme-theoretic intersection of every collection of its irreducible components is regular. We call such intersections {\em (closed) strata}. A scheme $X_0$ is a {\em strict normal crossings scheme} if every point of $X_0$ has a Zariski neighborhood $U$ and a closed immersion $i\colon U\to X'$ into some regular scheme such that $i(U)$ is a strict normal crossings divisor.  We will also need a notion of a transverse divisor $D$ in a strict normal crossings scheme $X_0$. 

 \begin{definition} \label{d:ncpair}
  A scheme $X_0$ and divisor $D\subset X_0$ form a {\em normal crossings pair} $(X_0,D)$ if every point of $X_0$ has a Zariski neighborhood $U$ and a closed immersion $i\colon U\to X'$ into some regular scheme such that
  \begin{enumerate}
    \item there exists divisors $Y'$ and $D'$ in $X'$ such that $Y'\cup D'$ is strict normal crossings in $X'$,
    \item $i(U)=Y'$ and $i(D\cap U)=D'\cap Y'$.
\end{enumerate}
\end{definition}

Observe that in this case, $D$ intersects the strata of $X_0$ transversely. We will conceptualize $D$ as a {\em divisor at infinity} and act as if the pair $(X_0,D)$ were the scheme $X_0\setminus D$. A pair $(X_0,D)$ can arise from the following situation: let $\cX$ be a semistable scheme over $\cO$ and $\mathcal{D}$ be a divisor on 
$\cX$ flat over $\cO$  such that $X_0\cup \mathcal{D}$ is a strict normal crossings divisor; then let $X_0$ and $D$ be the closed fibers of $\cX$ and $\cD$, respectively.

\begin{example} \label{e:snclog}
The simplest example of a normal crossings divisor has a natural log structure.
Consider $\Aff^n=k[x_1,\dots,x_n]$, and let $D\subset \Aff^n$ be the divisor cut out by $x_1x_2\cdots x_m=0$. The log structure $M_D$ on $\Aff^n$ attached to $D$ is obtained from the pre-log structure $\N^{m}\to \cO_{\Aff_n}$ where $\N^{m}$ is generated by $f_1,\dots,f_m$ with $f_i\mapsto x_i$. 
\end{example}

\begin{example} \label{e:semistablelog}
A log structure on the normal crossings scheme 
\[Y_{n,m}=\Spec k[x_0,\dots,x_n]/(x_0x_1\cdots x_m)\] 
is induced by $\N^{m+1}\to \cO_{Y_{n,m}}$ with $\N^{m+1}$ generated by $f_0,f_1,\dots,f_m$ with $f_i\mapsto x_i$.
This is naturally a log scheme over $S^\dagger$ with the morphism $Y_{n,m}\to S^\dagger$ induced by 
\[e\mapsto f_0+f_1+\dots+f_m.\]
This example arises by considering the morphism $X_n=\Spec k[x_0,\dots,x_n]\to \Spec k[t]$, given by $t\mapsto x_0\cdots x_m$ and taking the fiber over $S^\dagger$, considered as the log subscheme $\{t=0\}$. Suppose a divisor $D\subset X_n$ is transverse to coordinate subspaces given by intersecting subsets of  $\{x_0=0\},\dots,\{x_m=0\}$. Then $D$ induces a normal crossings pair $(Y_{n,m},D\cap Y_{n,m})$.
\end{example}

Locally, there is a natural log structure over $S^\dagger$  attached to normal crossing pairs $(X_0,D)$ \cite[Proposition~1.8.2]{Ogus:logbook}. Specifically, given $U\hookrightarrow X'$ together with $Y'$ and $D'$ as in Definition \ref{d:ncpair}, we can put the log structure $M_{X'}$ on $X'$ attached to the divisor $Y'\cup D'$.
Now, after possibly shrinking $U$, we may pick a function $t$ cutting out $Y'$, inducing a log morphism $(X',M_{X'})\to (\Aff^1,M_{\Aff^1})$ where $\Aff^1$ is given the log structure as in Example~\ref{e:snclog}. Here, the section $e$ of  $M_{\Aff^1}$ is mapped to $t$ in $\cO_U$.
Taking the inverse image log structures on $U$ and $0\in\Aff^1$, respectively, we obtain a morphism of log schemes $(U,M_U)\to S^\dagger$. While there are obstructions to globalizing this structure, we say $(X_0,D)\to S^\dagger$ is equipped with {\em the log structure of a strict normal crossings scheme with a divisor at infinity} if it is Zariski-locally locally isomorphic to this structure.


\begin{definition}
    Let $(X,D)$ be a smooth variety with a strict normal crossings divisor $D$. A subvariety $Z\subset X$ is {\em transverse-to-strata} if it is smooth and it intersects each stratum of $D$ transversely.
\end{definition}

\begin{definition}
  Let $(X_0,D)$ be a strict normal crossings pair. A closed subscheme $Z\subset X_0$ is {\em transverse-to-strata} if for any point of $Z$, there is a Zariski neighborhood $U$ in $X$ and a 
  closed immersion $i\colon U\to X'$ as in Definition~\ref{d:ncpair} such that there is a regular subscheme $Z'\subseteq X'$ such that 
  \begin{enumerate} 
    \item $Z'$ intersects the strata of $Y'\cup D'$ transversely, and 
    \item $i(Z\cap U)=Z'\cap Y'$.
  \end{enumerate}
\end{definition}

In the above, $Z$ is itself a strict normal crossings scheme. It is natural to give $Z$ the inverse image log structure $(Z,M|_Z)$ pulled back by the closed embedding $Z\to X_0$ from a log structure of a strict normal crossings scheme with a divisor at infinity. Observe that for a strict normal crossings pair $(X_0,D)$ and a transverse-to-strata subscheme $Z\subset X_0$, for any closed stratum $X_F$ of $X_0$, $X_F\cap Z$ is a transverse-to-strata subvariety of $(X_F, X_F \cap D)$.
The following is a straightforward local computation:
\begin{lemma} \label{r:inducedlogstructure}
   Let $(X,D)$ (resp.,~$(X_0,D)$) be a regular scheme with a normal crossings divisor (resp.,~a normal crossings pair), and let $Z\subseteq X$ (resp.,~$Z\subseteq X_0$) be a closed transverse-to-strata subscheme. Let $M$ be the log structure relative to $D$ (resp.,~of a strict normal crossings scheme with divisor at infinity). Then the inverse image log scheme $(Z,M|_Z)$  has the log structure attached to the normal crossings divisor $D\cap Z$ (resp.,~a normal crossings log structure with divisor $D\cap Z$ at infinity). 
\end{lemma}

All our examples will arise in the above fashion as transverse-to-strata subschemes of smooth toric varieties (resp.~toric schemes) and thus be log smooth over $S$ (resp.~$S^\dagger$).



\subsection{Log differentials}

\begin{definition} Let $f\colon  (X,M_X)\to (Y,M_Y)$ be a morphism of log schemes where $\varphi\colon f^{-1}M_Y\to M_X$ is the morphism of sheaves of monoids. The {\em sheaf of log differentials of $(X,M_X)$ over $(Y,M_Y)$}, denoted $\Omega^1_{(X,M_X)/(Y,M_Y)}$, is defined to be the $\cO_X$-module 
\[\Omega^1_{(X,M_X)/(Y,M_Y)}\coloneqq (\Omega^1_{X/Y}\oplus (\cO_X\otimes_{\Z} M_X))/K\]
where $K$ is the $\cO_X$-submodule generated by 
\[(d\alpha_X(a),0)-(0,\alpha_X(a)\otimes a)\text{ and }(0,1\otimes \varphi(b))\]
for $a\in M_X$ and $b\in f^{-1}M_Y$.
\end{definition}

By construction, there is a natural morphism of $\cO_X$-modules $\Omega^1_{X/Y}\to \Omega^1_{(X,M_X)/(Y,M_Y)}$.

\begin{rem}
    Elements of the form $(0,1\otimes a)$ can be written as $\DLog(a)$ and interpreted as $\frac{d\alpha_X(a)}{\alpha_X(a)}$ where the above relation imposes $a\DLog(a)=d\alpha_X(a).$ Note that $\DLog(a)$ makes sense even when $\alpha(a)$ is not a unit.
\end{rem}

\begin{rem}
    In Example~\ref{e:snclog}, the sheaf of log differentials of $(\Aff^n,M_D)$ over the trivial point $S$, $\Omega^1_{(\Aff^n,M_D)/S}$ is the locally free sheaf generated by 
    \[\frac{dx_1}{x_1},\dots,\frac{dx_m}{x_m},dx_{m+1},\dots,dx_n,\]
    giving a local model for working relative to a simple normal crossing divisor $D$. Such sheaves of log differentials are denoted by $\Omega^1(\log D)$.

    In Example~\ref{e:semistablelog}, the sheaf of log differentials over the standard log point, $\Omega^1_{Y_{n,m}/S^\dagger}$ is 
    the quotient of the locally free sheaf generated by $\frac{dx_0}{x_0},\dots,\frac{dx_m}{x_m},dx_{m+1},\dots,dx_n$ by the subsheaf generated by 
    \[\frac{dx_0}{x_0}+\dots+\frac{dx_m}{x_m}.\]
\end{rem}

We have the following observations that can be proved locally:
\begin{lemma} \label{l:transverselogstructure}
Let $(X_0,M)$ be a log scheme with the log structure of a strict normal crossings scheme with divisor at infinity $(X_0,D)$.
\begin{enumerate}
    \item If $j\colon X_F\to X_0$ is the closed immersion of a stratum, equipped with the inverse image log structure $(X_F,M_F)$, then there is a canonical isomorphism,
    \[\Omega^1_{(X_F,M_F)/S^{\dagger}}\cong j^*\Omega^1_{(X_0,M)/S^\dagger}. \]
    
    \item If $i\colon Z\hookrightarrow X_F$ is a closed transverse-to-strata subscheme of $X_F$ (where $(X_F,M_F)$ is the inverse image log structure as above) and $Z$ is given the inverse image log structure $(Z,M_Z)$, then $\Omega^1_{(Z,M_Z)/S^{\dagger}}$ is the pushout in the following diagram of locally free sheaves on $Z$
     \[\xymatrix{i^*\Omega^1_{X_F/S}\ar[r]\ar[d]&i^*\Omega^1_{(X_F,M_F)/S^\dagger}\ar[d]\\
     \Omega^1_{Z/S}\ar[r]&\Omega^1_{(Z,M_Z)/S^\dagger}.}\]
     \end{enumerate}
\end{lemma}

The second part of this lemma shows that the log differentials on $Z$ are determined by the pullback of usual differentials from $X_F$ to $Z$ and the morphism $\Omega^1_{X_F/S}\to \Omega^1_{(X_F,M_F)/S^\dagger}$. 
An analogous result holds for $\Omega^1_{(Z,M_Z)/S}$ when $(X,D)$ is a regular scheme over $S$ with a normal crossings divisor $D$ and $Z\subseteq X$ is transverse-to-strata.

Log differentials will enter into the definition of de Rham fundamental groups according to the following notion:

\begin{definition}
  Let $f\colon X\to Y$ be a morphism of scheme. A {\em shortened complex of differential forms} $\tilde{\Omega}$ on $X$ over $Y$ is the data of a commutative diagram whose entries are $\cO_X$-modules
  \[\xymatrix{
  0\ar[r]&\cO_X\ar[r]^>>>>>{d}&\Omega_{X/Y}^1\ar[d]^{j}\ar[r]^d &\Omega_{X/Y}^2\ar[d]^{j\wedge j}\\
         &             &\tilde{\Omega}^1\ar[r]^{d_{\tilde{\Omega}}} &\tilde{\Omega}^2}
  \]
  where the horizontal maps are morphisms of $f^{-1}\cO_Y$-modules and the vertical maps are morphisms of $\cO_X$-modules, 
$\Omega^1_{X/Y}\to \tilde{\Omega}^1$ is an injective morphism, $\tilde{\Omega}^2=\bigwedge^2\tilde{\Omega}^1$, and $d_{\tilde{\Omega}}$ satisfies the Leibniz rule
\[d_{\tilde{\Omega}}(f\theta)=fd_{\tilde{\Omega}}\theta+j(df)\wedge\theta\]
where $f$ and $\theta$ are local sections of $\cO_X$ and $\tilde{\Omega}^1$, respectively.
  A {\em morphism of shortened complexes} $\tilde{\Omega}\to \tilde{\Upsilon}$ is induced by a morphism of $\cO_X$-modules $\tilde{\Omega}^1\to {\tilde{\Upsilon}}^1$ such that the relevant diagram commutes.
\end{definition}

This is the part of the log de Rham complex (\cite[Section~V.2]{Ogus:logbook}) that is relevant to the construction of the  unipotent fundamental group. The natural example of such a shortened complex is $\Omega^1_{(X,M_X)/S^\dagger}\to \Omega^2_{(X,M_X)/S^\dagger}$ for $(X,M_X)\to S^\dagger$ a morphism of fine saturated log schemes. 

We can define pullbacks of shortened complexes by the pushout diagram in Lemma~\ref{l:transverselogstructure}:

\begin{definition} \label{pullbackscodf}
    Let $X\to Y$ be a smooth morphism with a shortened complex of differential forms $\tilde{\Omega}$. For a smooth subvariety $i\colon Z\hookrightarrow X$, we say a shortened complex of differential forms $\tilde{\Upsilon}$ on $Z$ over $Y$ is the {\em pullback of $\tilde{\Omega}$ by $i$} if
   $\tilde{\Upsilon}^1$ is the pushout
         \[\xymatrix{i^*\Omega^1_{X/Y}\ar[r]\ar[d]&i^*\tilde{\Omega}^1\ar[d]\\
     \Omega^1_{Z/Y}\ar[r]&\tilde{\Upsilon}^1}\]
   and $d_{\tilde{\Upsilon}}\colon\tilde{\Upsilon}^1\to \tilde{\Upsilon}^2$ is induced by $i^*{d_{\tilde{\Omega}}}$ and $d\colon \Omega^1_{Z/Y}\to \Omega^2_{Z/Y}$.
\end{definition}

\subsection{Log structures on toric varieties} 

A toric variety $\PP(\Delta)$ arises from a rational fan $\Delta$ in $N_\R$ \cite{CLS}, is smooth if $\Delta$ is unimodular, and is proper if the support of $\Delta$ is $N_\R$.
We will use the data of a partial compactification to put a log structure on $\PP(\Delta)$ according to the following lemma whose proof is a straightforward verification.

\begin{lemma}
    Let $(\Delta,\on{Id}\colon N\to N,R)$ be a trivially enriched fan with partial compactification. We define a presheaf of monoids $P'$ on $\PP(\Delta)$, as follows:
    for $U\subset \PP(\Delta)$, if $U$ is not contained in a toric affine open, set $P'(U)=\{0\}$; otherwise let $U_\sigma=\Spec k[\sigma^\vee]$ be the minimal toric affine open containing $U$, and set 
    \[P'(U)=\left(\sigma^\vee \cap \bigcap_{\rho\in R\cap \sigma} \rho^\perp \right)\cap M.\] 
    There is a morphism $\alpha\colon P'\to \cO_X$ induced by the natural inclusion
    \[\alpha\colon P'(U_\sigma)\hookrightarrow k[\sigma^\vee]\]
   Let $P\to \cO_X$ be the sheafification of the above presheaf and morphism.
    Then $(P,\alpha)$ defines a pre-log structure on $\PP(\Delta)$.
\end{lemma}

The log structure induced by this pre-log structure is {\em the log structure attached to} $(\Delta,\on{Id}\colon N\to N,R)$. Morphisms of trivially enriched fans with partial compactifications induce morphisms of log schemes.

 \begin{example}
  We have the following natural examples:
  \begin{enumerate}
     \item the usual toric log structure when $R=\varnothing$, and
     \item the trivial log structure when $R$ consists of all rays of $\Delta$ and  $ \Delta $ is full-dimensional (in which case $P(U)=\{0\}$ for all $U\subset \PP(\Delta)$).
  \end{enumerate}
  The usual log structure gives a sheaf of log $1$-forms on $U_{\sigma}$ obtained by adjoining to the usual $1$-forms, $\cO_{\PP(\Delta)}$-linear combinations of symbols $\DLog(m)$ for $m\in \sigma^\vee\cap M$  subject to the relations
    \begin{enumerate}
    \item $\alpha(m)\DLog(m)=d(\alpha(m))$ (where $\alpha(m)$ is viewed as a cocharacter on the open torus in $U_{\sigma}$), and
    \item $\DLog(m_1m_2)=\DLog(m_1)+\DLog(m_2)$
    \end{enumerate}
  where the operation in monoids is written multiplicatively. We will extend $\DLog$ linearly to make sense of $\DLog(m)$ for $m\in M_{\C}$.
  If  $\PP(\Delta)$ is a smooth toric variety with its usual log structure, the sheaf of log differentials is $\Omega^1(\log D)$ where $D$ is the union of the toric divisors.
\end{example}

We will now define sheaves of differentials arising from enriched fans with partial compactifications. Here, the enriched fans alow us to enlarge $M$, matching the situation when the log structure on a toric variety is the inverse image structure of its inclusion as a torus orbit in a larger toric variety.
Recall that for a cone $\tau\in\Delta$, the orbit closure $V(\tau)$ is isomorphic to the toric variety with fan given by the star-quotient $\Delta_\tau$. When $\rho$ is a ray of $\Delta$, we may use $D_\rho$ to denote $V(\rho)$.  
The motivation is the following lemma which follows from definitions:
\begin{lemma}
    Let $(\Delta,\on{Id}\colon N\to N,R)$ be a trivially enriched fan with partial compactification such that $\Delta$ is unimodular and $R$ is a collection of rays of $\Delta$. Then, the log structure on $\PP(\Delta)$ is equivalent to the log structure induced by the normal crossings divisor $\cup_{\rho\not\in R} D_\rho$, thus is fine saturated.
\end{lemma}

We can define a sheaf and ultimately a shortened complex of differential forms from an enriched fan with partial compactification. It will be isomorphic to  a sheaf of log differentials.

\begin{definition}
  Let $(\Delta_1,\pi_{1}\colon N_1\to N'_1,R_1)$ be an enriched fan with partial compactification.
  For a morphism of enriched fans 
  \[f\colon (\Delta_1,\pi_{1}\colon N_1\to N'_1)\to (\Delta_2,\pi_{2} \colon N_2\to N'_2),\]
  with $f(R)=\{0\}$,
  we will define the sheaf $\Omega^1_{(\Delta_1,N_1 \to N'_1,R_1)/(\Delta_2,N_2\to N'_2,\varnothing)}$ (abbreviated $\Omega^1$) of {\em log differentials attached to $f$} 
  on $\PP(\Delta_1)$.
  For a cone $\sigma$ of $\Delta_1$, set 
  \[V_{1,\sigma}\coloneqq \bigcap_{\rho\in \pi_1^{-1}(\sigma) \cap R_{1}} \rho^\perp\subseteq M_{1,\R}.\]  
  The sheaf $\Omega^1$ on $\PP(\Delta_1)$ is the sheafification of the presheaf given as follows:
  for $U\subset\PP(\Delta)$, if $U$ is not contained in any toric affine open, $\Omega^1(U)\coloneqq\Omega^1_{\PP(\Delta_1)/\PP(\Delta_2)}(U)$; otherwise, if $U_\sigma$ is the minimal toric affine open containing $U$, then
  \[\Omega^1(U)\coloneqq
   (\Omega^1_{\PP(\Delta_1)/\PP(\Delta_2)}(U)\oplus (\cO_U\otimes_{\Z} (V_{1,\sigma}\cap M_1)))/K\]
where $K$ is the $\cO_X$-submodule generated by 
\[(d\alpha_{\PP(\Delta_1)}(m),0)-(0,\alpha_{\PP(\Delta_1)}(m)\otimes m)\]
for $m\in V_{1,\sigma}\cap \sigma^\vee\cap M'_1$ and by
$(0,1\otimes f(b))$ for $b\in M_2$. Set $\Omega^2=\bigwedge^2\Omega^1$. Write $j\colon\Omega^*_{\PP(\Delta_1)/\PP(\Delta_2)}\to \Omega^*$ for the morphism induced by inclusion on the first factor.
 The differential $d_{\Omega}\colon \Omega^1\to \Omega^2$ is defined by
\[d_\Omega(\theta,g\otimes m)=j(d\theta)+j(dg)\wedge (0,m).\]
\end{definition}

The above definition gives a shortened complex of differential forms.

\begin{rem}

The essential idea in the above is that for $m\in V_{1,\sigma}\cap M_1$, we adjoin $a\DLog(m)$. We then quotient by the relation $\alpha(m)\DLog(m)=d\alpha(m)$ and by pullbacks from the base.
Observe that in the above, the first set of relations come from elements of $M'_1$, not $M_1$. The motivation is the following. If one has an unenriched fan $(\Delta,\on{Id}\colon N\to N,R)$ with star-quotient $(\Delta_\tau,N\to N'\coloneqq N/N_\tau,R')$, then $m\in \tau^\vee$ restricts to a rational function on $V(\tau)$. If $m\in\tau^\vee\setminus M'_1$, then $m\not\in\tau^\perp$ and $m$ restricts to $0$ on $V(\tau)$. Thus, we have a log differential $\DLog(m)$ where the relation $\alpha(m)\DLog(m)=d\alpha(m)$ becomes the trivial relation $0\DLog(m)=d0$.
\end{rem}



\begin{rem} \label{r:maximalsubspacedata}
We can get a sense on log differentials by comparing them to an empty partial compactification.
Let $\Delta'$ be $\Delta_S$ or $\Delta_{S^\dagger}$. Let 
\[f\colon (\Delta,\pi\colon N\to N',R)\to \Delta'\]
be a morphism of enriched fans with partial compactifications.
We have an injective morphism of sheaves
\[\Omega^1_{(\Delta,\pi\colon N\to N',R)/\Delta'}\hookrightarrow \Omega^1_{(\Delta,\pi\colon N\to N',\varnothing)/\Delta'}\cong (M/f^*M_{\Delta'})\otimes \cO_{\PP(\Delta)}.\]
\end{rem}

The following lemma says that log differentials on strata can be recovered from the enriched fan by the above construction.

\begin{lemma} \label{l:reldifferentials}
  Let 
  $f\colon (\Delta_1,N_1\to N_1,R)\to (\Delta_2,N_2\to N_2,\varnothing)$
  be a morphism of trivially enriched fans with partial compactifications inducing a morphism of log toric varieties $(\PP(\Delta_1),M_{\PP(\Delta_1)})\to (\PP(\Delta_2),M_{\PP(\Delta_1)})$. 
  
    Let $\tau_1\in \Delta_1,\tau_2\in\Delta_2$ such that $\tau_2$ is the minimal cone of $\Delta_2$ containing $f(\tau_{1})$ so there is a morphism ennriched fans with partial compactifications and of log schemes 
    \begin{align*}
    (\Delta_1)_{\tau_1}&\to (\Delta_2)_{\tau_2}\\
    f\colon (V(\tau_1),M_{V(\tau_1)})&\to (V(\tau_2),M_{V(\tau_2)})
    \end{align*}
  where each orbit closure is given the inverse image log structure induced by inclusion.
 
 Then, there is a canonical isomorphism of sheaves of $\cO_{V(\tau_{1})}$-modules,
  \[\Omega^1_{(V(\tau_1),M_{V(\tau_1)})/(V(\tau_2),M_{V(\tau_2})}\cong \Omega^1_{(\Delta_1)_{\tau_1}/(\Delta_2)_{\tau_2}}\]
  that extends to an isomorphism of shortened complexes.
  This isomorphism commutes with the restriction map induced by the inclusion $\tau_1\subset \tau'_1$ when $\tau_2$ is the minimal cone in $\Delta_2$ for both $f(\tau_1)$ and $f(\tau_{1}')$. 
\end{lemma}

\begin{proof}
  Let $\sigma$ be a cone of $\Delta$ containing $\tau_1$ corresponding to $\overline{\sigma}$ in $(\Delta_1)_{\tau_1}$. 
   We will verify the isomorphism for the presheaves inducing each sheaf on the affine open $U_{\overline{\sigma}}$ in $\PP((\Delta_1)_{\tau_1})$.
  The log structure on $(V(\tau_1),M_{V(\tau_1)})$ is induced on $U_{\overline{\sigma}}$ by 
  \[P_\sigma\to k[\sigma^\vee]\to k[\sigma^\vee\cap\tau^\perp]\]
  where 
  \[P_\sigma=\left(\sigma^\vee \cap \bigcap_{\rho\in R\cap \sigma} \rho^\perp \right)\cap M.\]
  Then, on $U_{\overline{\sigma}}$,
  \[\Omega^1_{(V(\tau_1),M_{V(\tau_1)})/(V(\tau_2),M_{V(\tau_2)})}\coloneqq (\Omega^1_{V(\tau_1)/V(\tau_2)}\oplus (\cO_X\otimes_{\Z} P_{\sigma}))/K\]
  where $K$ is generated by the  relations coming from $(V(\tau_2),M_{V(\tau_2)})$ and from $a\DLog(a)=da$. It is immediate that this coincides exactly with the presheaf used to define $\Omega^1_{(\Delta_1)_{\tau_1}/(\Delta_2)_{\tau_2}}$.
  The isomorphism of shortened complexes is a straightforward verification.
\end{proof}

\begin{example}
  A natural example is the morphism $\Aff^1\to \{\on{pt}\}$ induced by the morphism of partial compactifications
  \[(\Delta_{\Aff^1},\Z\to \Z,\varnothing)\to \Delta_S\]
  where $\Delta_{\Aff^1}=\{0,\R_{\geq 0}\}.$
  Setting $\tau_1=\R_{\geq 0}$ and $\tau_2=\{0\}$, we get the induced morphism of stars
  $f\colon \Delta_{S^\dagger}\to\Delta_S$ and of log schemes $S^\dagger\to S$. The sheaf of log differentials $\Omega^1_{S^\dagger/S}$ is rank $1$, generated by $\DLog(e)$. This differential corresponds to $\frac{dt}{t}$ where $t$ is a coordinate on $\Aff^1$.
\end{example}

\subsection{Toric Schemes}

We will review toric schemes over a discrete valuation ring \cite{KKMS} following the exposition of \cite{NS,HelmKatz} 
Let $\cO$ be a discrete valuation ring with residue field $k$, field of fractions $K$, and uniformizer $\pi$.

Let $\Sigma$ be a completable rational polyhedral complex in $N_\R$. 
We can construct a log toric scheme over $S^\dagger$ from a rational polyhedral complex with partial compactification $(\Sigma,N,R)$ where $R$ is a set of rays of the recession fan of $\Sigma$. Let $\tSigma$ be the cone over $\Sigma$, considered as a fan in $N_\R\times\R$. By identifying $N_{\R}$ with $N_{\R}\times\{0\}$, we view $R$ as a partial compactification of $\tilde{\Sigma}$. Then, $(\tSigma,\on{Id}\colon N\times\Z\to N\times \Z,R)$  defines a log scheme 
$\PP(\tSigma)_{\ZZ}$ over $\Z$.

The projection $N\times \Z\to\Z$ induces a map of enriched fans with partial compactifications
\[(\tSigma,\on{Id}\colon N\times\Z\to N\times\Z,R)\to \Delta_{\Aff^1}=(\{0,\R_{\geq 0}\},\on{Id}\colon\Z\to \Z,\varnothing).\]
giving rise to a log morphism
\[\pi_{\Z}: \PP(\tSigma)_{\ZZ} \rightarrow \Aff^1_{\ZZ}.\]
Because $\Aff^1_{\ZZ}=\Spec \ZZ[t]$ has the log structure induced by the monoid $\N$ with $e\mapsto t$, there is a natural log morphism $\iota_{S^\dagger}\colon S^\dagger\to \Aff^1_{\ZZ}$ induced by a morphism of monoids $e\mapsto e$. Similarly, there's a morphism of log schemes $\Spec \cO\to \Aff^1$ induced by $t\mapsto \pi$.

\begin{definition}
  Define the log schemes $\cP(\Sigma)_{\cO}$ and $\cP(\Sigma)_0$ by
  \[\cP(\Sigma)_{\cO}\coloneqq \PP(\tSigma)\times_{\Aff^1} \cO,\ \cP(\Sigma)_0\coloneqq \PP(\tSigma)\times_{\Aff^1} S^\dagger.\]
\end{definition}

We can view $\cP(\Sigma)_0$ as the closed fiber of $\cP(\Sigma)_{\cO}$ with the inverse image log structure.
The generic fiber $\cP(\Sigma) \times_{\Spec \cO} \Spec K$ is isomorphic
to the toric variety over $K$ attached to the recession fan $\Sigma_0$ with trivial enriched structure $N\to N$ and partial compactification $R$. 

\begin{rem}
If $\Sigma$ is {\em integral}, i.e.,~the vertices of every polyhedron in $\Sigma$
lie in $N$, then the closed fiber of $\cP(\Sigma)_{\cO}$
is reduced.   
Adjoining a $d$th root of the uniformizer $\pi$ to $\cO$ has the effect
of rescaling $\Sigma$ by $d$; that is, if $\cO^{\prime} = \cO[\pi^{\frac{1}{d}}]$, then
the base change of the family $\cP(\Sigma) \rightarrow \cO$ is the family 
$\cP(d\Sigma)) \rightarrow \cO^{\prime}$. Hence, we can ensure that the closed fiber is reduced by base change.
\end{rem}

Because we will be  interested in degenerations of toric varieties whose special
fiber is a divisor with simple normal crossings, we will make use of the following slight extension of a deep result from \cite{KKMS}: 

\begin{prop}\cite[Prop.~2.3]{HelmKatz} \label{p:nc} 
Let $\Sigma$ be a complete rational polyhedral complex in $N$. 
There exists an integer $d$ and a subdivision $\Sigma^{\prime}$ of $d\Sigma$ such that the generic
fiber of the scheme $\cP(\tSigma^{\prime})_{\cO}$ is a smooth toric variety
and the closed fiber of $\PP(\tSigma^{\prime})$ is a divisor with simple normal crossings.  Moreover, if the recession fan $\Sigma_0$ is already simplicial and unimodular, $\Sigma^{\prime}$ can be chosen to have $\Sigma_0'=\Sigma_0$.
\end{prop}

The structure of $\cP(\Sigma)_0$ is described by the following which follows from definition:

\begin{lemma}
    Let $\Sigma$ be a completable unimodular rational polyhedral complex in $N_\R$. Let $R$ be a set of rays of the recession cone $\Sigma_0$. Then, the closed fiber $\cP(\Sigma)_0$ is a log scheme given as a strict normal crossings scheme with divisor at infinity $\bigcup_{\rho\not\in R} (D_{\rho}\cap \cP(\Sigma)_0)$.
\end{lemma}

We can probe the closed fiber by using the inclusion-reversing bijection
between closed torus orbits in $\cP(\Sigma)_{0}$ and polyhedra $P\in\Sigma$. 
To a polyhedron $P\in\Sigma$, {\em the closed orbit $(V(P),M_{V(P)})$} in $\cP(\Sigma)_{0}$ attached to $P$ is
$(V(\tilde{P}),i_{V(\tilde{P})}^*M_{\PP(\tSigma)})\times_{\Aff^1} S^\dagger$,
that is, the log scheme given the inverse image log structure by $i_{V(\tilde{P})}\colon V(\tilde{P})\to \PP(\tSigma)$ base-changed to  $S^\dagger$. Thus, the components of $\cP(\Sigma)_0$ correspond to vertices of $\Sigma$.
The underlying scheme of $V(P)$ is the toric variety $\PP(\tSigma_{\tilde{P}})=\PP(\Sigma_P)$. The sheaf of log differentials on $V(P)$ is described by an enriched fan with partial compactification according to the following:


\begin{lemma} \label{l:logstructureonorbits} 
  Let $(\Sigma,N,R)$ be a complete polyhedral complex with partial compactification inducing a log structure on $\cP(\Sigma)_0$. Let $(V(P),M_{V(P)})$ be the closed torus orbit.
Then, the induced maps of star-quotients
  \begin{align*}
  (\tilde{\Sigma}_{\tilde{P}},N\times \Z\to (N\times\Z/(N\times\Z)_{\tilde{P}}),R')&\to \Delta_{S^\dagger}\\ 
  (\Sigma_P,N\to N/N_P,R')&\to \Delta_S
  \end{align*}
  (with domains abbreviated as $\tilde{\Sigma}_{\tilde{P}}$ and $\Sigma_P$, respectively)
  induce canonical isomorphisms of $\cO_{V(P)}$-modules,
  \[\Omega^1_{(\PP(\Sigma_P),M_{\PP(\Sigma_P)})/S^\dagger}\cong
  \Omega^1_{\tilde{\Sigma}_{\tilde{P}}/\Delta_{S^\dagger}}\cong
  \Omega^1_{\Sigma_P/\Delta_S}\]
  inducing isomorphisms of shortened complexes. Moreover, for an inclusion of polyhedra $P\subset Q$ of $\Sigma$ inducing $i\colon V(Q)\hookrightarrow V(P)$, the following diagram  of sheaves of $\cO_{V(Q)}$-modules (whose arrows are isomorphims) commutes:
 \[\xymatrix{
 i^*\Omega^1_{(\PP(\Sigma_P),M_{\PP(\Sigma_P)})/S^\dagger}\ar[r]\ar[d]&
 i^*\Omega^1_{\tilde{\Sigma}_{\tilde{P}}/\Delta_{S^\dagger}}\ar[r]\ar[d]&
i^*\Omega^1_{\Sigma_P/\Delta_S}\ar[d]\\
 \Omega^1_{(\PP(\Sigma_Q),M_{\PP(\Sigma_Q)})/S^\dagger}\ar[r]&
 \Omega^1_{\tilde{\Sigma}_{\tilde{Q}}/\Delta_{S^\dagger}}\ar[r]&
  \Omega^1_{\Sigma_Q/\Delta_S}.\\
  }\]
\end{lemma}

\begin{proof}
  Apply Lemma~\ref{l:reldifferentials} to the morphism
  \[(\tilde{\Sigma},N\times \Z\to N\times \Z,R)\to (\Delta_{\Aff^1},\Z\to\Z,\varnothing)\]
  and the cone $\tilde{P}\in\tilde{\Sigma}$ to obtain the isomorphism $\Omega^1_{(\PP(\Sigma_P),M_{\PP(\Sigma_P)})/S^\dagger}\cong
  \Omega^1_{\tilde{\Sigma}_{\tilde{P}}/\Delta_{S^\dagger}}$
   By unpacking definitions, one observes that the sheaves $\Omega^1_{\tilde{\Sigma}_{\tilde{P}}/\Delta_{S^\dagger}}$ and $\Omega^1_{\Sigma_P/\Delta_S}$ are canonically isomorphic. The commutativity of the diagram is a straightforward verification.

\end{proof}

\subsection{Subschemes of toric schemes}

We will study subvarieties of an algebraic torus, $X\subseteq (K^*)^n$, where $K$ is a discretely valued field. It is desirable to compactify the ambient space $(K^*)^n$ to a toric scheme $\cP(\Sigma)$ in such a way that $\cX\coloneqq \overline{X}$ is well-behaved. 
We will consider the case where  $X$ is  {\em sch\"{o}n} as introduced by Tevelev \cite{Tevelev:compactifications} which is automatic when $X$ has a smooth tropicalization.  Here, we will follow the exposition in \cite{LuxtonQu,HelmKatz} and also recommend \cite{Gubler:guide}.

\begin{definition} A subvariety $X\subseteq (K^*)^n$ is {\em sch\"{o}n} if there exists a toric scheme $\cP(\Sigma)_{\cO}$ such that the closure $\cX=\overline{X}$ satisfies
\begin{enumerate}
    \item the projection $\cX\to \Spec \cO$ is proper, and
    \item the group action $(\Gm)^n\times_{\cO} \cX\to \cP(\Sigma)_{\cO}$ is smooth.
\end{enumerate}
\end{definition}

The following is a straightforward extension of the analogous result in \cite{Tevelev:compactifications} (see also \cite{Hacking:homology}):

\begin{prop} \label{p:dilatesnc} Let $X\subseteq (K^*)^n$ be a sch\"{o}n subvariety, and let $\Delta$ be a projective rational fan in $\R^n$. Then there exists a positive integer $d$ and a toric scheme $\cP(\Sigma)$ (with $\Sigma$ complete and unimodular) over $\cO'=\cO[\pi^{\frac{1}{d}}]$ such that for $K'=K[\pi^{\frac{1}{d}}]$ and $\cX=\overline{X}$,
\begin{enumerate}
    \item \label{i:recfan} the recession fan of $\Sigma$ is a subfan of a refinement of $\Delta$,
    \item the generic fiber $\cP(\Sigma)_{K'}$ is a smooth toric variety,
    \item the closed fiber $\cP(\Sigma)_0$ is a strict normal crossings scheme,
    \item the generic fiber $\cX_{K'}$ is proper and smooth, and $\overline{X}\setminus X$ is a strict normal crossings divisor, and
    \item the closed fiber $X_0$ of $\cX$ is a proper normal crossings scheme that is transverse-to-strata in $\cP(\Sigma)_0$.
\end{enumerate}
\end{prop}

\begin{proof}
  Except for condition (eqref{i:recfan}, this is \cite[Proposition~3.10]{HelmKatz}. We explain how to modify the proof. By the combinatorial description of ample line bundles \cite{CLS}, there is a piecewise linear function $h_{\Delta}$ that is strictly convex on $\Delta$, i.e.,~the cones of $\Delta$ are exactly $h_{\Delta}$'s domains of linearity. As a step in \cite[Proposition~3.10]{HelmKatz}, we construct $\Sigma_1$, the Gr\"{o}bner complex of $X$ \cite{Gubler:guide} which also produces a piecewise affine function $h_{\Sigma_1}$ that is strictly convex on $\Sigma_1$. Hence, $h_{\Delta}+h_{\Sigma_1}$ induces a common refinement $\Sigma'$ of $\Sigma_1$ and $\Delta'$. Now, the steps in \cite{HelmKatz} will produce $\Sigma$ as a subcomplex of a refinement of $\Sigma'$.
\end{proof}

In all of our cases below, we will let $R$ be a subset of the rays of the recession fan $\Sigma_0$.
We give $\cX$ the inverse image log structure by $\cX\hookrightarrow\PP(\Sigma)_{\cO}$. 
Then, by Remark~\ref{r:inducedlogstructure}, $\cX$ will have the log structure induced by the intersection of $\cX$ with the divisor $\cP(\Sigma)_0+\sum_{\rho_0\in (\Sigma_0)_{(1)}\setminus R} D_{\rho_0}$
where $\cP(\Sigma_0)$ is the closed fiber, considered as a divisor.
We give the closed fiber $X_0$, the inverse image log structure by $X_0\hookrightarrow\cX$. Now, Lemma~\ref{l:transverselogstructure} applies to the log structure on $X_0$.

The following lemma allow us to compare two shortened complexes on the stratum $X_F$.

\begin{lemma} \label{l:logstructureontransversetoorbits} 
  Let $(\Sigma,N,R)$ be a unimodular complex with enriched structure and partial compactification. Let $(\cP(\Sigma)_0,M)$ be the attached log scheme over $S^\dagger$.
  Let $P$ be a polyhedron of $\Sigma$. Let $X_0$ be a transverse-to-strata subscheme of $\cP(\Sigma)_0$. Then there is a canonical isomorphism of shortened complexes of differential forms on $X_P$ between    
  \begin{enumerate}
    \item the pullback complex induced by $X_P\to \cP(\Sigma)_0$ over $S^\dagger$ where $\cP(\Sigma)_0$ has log structure $M$ and 
    \item the pullback complex induced by $X_P\to \PP(\Sigma_P)$ over $S$ where the shortened complex  on $\PP(\Sigma_P)$ is induced by $(\Sigma_P,N\to N/N_P,R')$ as in Definition~\ref{d:starquotientcomplex}. 
  \end{enumerate}
  The isomorphism commutes with pullback induced by $V(Q)\hookrightarrow V(P)$ for $P\subseteq Q$.
\end{lemma}

\begin{proof}
  One first considers the case where $X_0=\cP(\Sigma)_0$ and uses Lemma~\ref{l:logstructureonorbits} to identify the two complexes on $V(P)$. Then one pulls back by $X_P$ by using Lemma~\ref{l:transverselogstructure}.  
\end{proof}

\subsection{Tropicalization of Subvarieties}

Given a subvariety  $X\subset (K^*)^n$, tropicalization \cite{MaclaganSturmfels,Gubler:guide} produces a tropical variety $\Trop(X)\subset \R^n$. The tropical variety is always supported on a complete rational polyhedral complex called a Gr\"{o}bner complex. In this way, it obtains a non-canonical structure of a balanced weighted rational polyhedral complex. Additionally, the recession cones of $\Trop(X)$ form a fan.

When $K$ is a field with trivial valuation, tropicalizations are always fans. If the closure in a toric variety $\overline{X}\subset \PP(\Delta)$ is transverse-to-strata, then $\Trop(X)$ is supported on $\Delta$.

Varieties whose tropicalizations have support equal to Bergman fans (with weight $1$ on all top-dimensional cones) are well-understood: they are subspaces. Indeed, let $(K^*)^n$ denote the open torus in the projective space $\PP_K^n$.
By \cite[Section~9]{Sturmfels:solving}, for a projective subspace $X\subset \PP_K^n$ not contained in any coordinate hyperplane with matroid $\M$, the tropicalization $\Trop(X\cap (K^*)^n)$ is the support of the Bergman fan $\Delta_\M$ where $\M$ is the matroid attached to $X$ and where every top-dimensional cone is given multiplicity $1$. The converse is a result due to Mikhalkin based on work with Sturmfels and Ziegler and written down in \cite[Prop~4.2]{KP:Realization}:

\begin{prop} \label{p:matroidtrop}
    Let $X\subset (K^*)^n$ be a subvariety. Suppose that $\Trop(X)$ is the Bergman fan of a matroid $\M$, all of whose top-dimensional cones have weight $1$. Then the closure of $X$ in $\PP^n$ is a projective subspace with matroid $\M$.
\end{prop}

\begin{definition}
Let $K$ be an arbitrary valued field. An algebraic variety $X\subseteq (K^*)^n$ is {\em tropically smooth} \cite{KS:TMNF,Shaw:matroidal} if $\Trop(X)$ is smooth (Definition~\ref{d:smoothcomplex}) with respect to some polyhedral structure.
\end{definition}

\begin{theorem} \cite[Corollary~10]{Cartwright:Grobner}
If $X\subseteq (K^*)^n$ is tropically smooth, then $X$ is sch\"{o}n.
\end{theorem}

By combining results from this section:

\begin{prop} \label{p:schonlog}
Let $X\subset(K^*)^n$ be a tropically smooth variety. Let $\Delta$ be a projective fan. Then, there is a rational polyhedral complex $\Sigma$ as in Proposition~\ref{p:dilatesnc} where $\cX$ is the closure of $X$ in $\cP(\Sigma)_{\cO'}$ with $\cO'=\cO[\pi^{\frac{1}{d}}]$. Let $R$ be a set of rays of the recession fan $\Sigma_0$ of $\Sigma$, and give $\cP(\Sigma)_{\cO}$ the log structure induced from $(\Sigma,N\coloneqq\Z^n,R)$.  
The closed fiber $X_0$ is transverse-to-strata in $(\cP(\Sigma)_0,D)$ where $D$ is the divisor at infinity, $D\coloneqq \bigcup_{\rho\in(\Sigma_0)_{(1)}\setminus R}(D_\rho\cap \cP(\Sigma)_0)$. Moreover, the inverse image log structure on $X_0$ is that of a simple normal crossings scheme with divisor $D\cap X_0$ at infinity. Moreover, each stratum $X_F$ is connected.
\end{prop}

It is an observation of Tevelev \cite{Tevelev:compactifications,LuxtonQu} that one can choose $\Sigma$ such that its support is $\Trop(X)$. 

\section{Logarithmic Unipotent Fundamental Group} \label{s:logpi1}

\subsection{Definitions}
We will recall the theory of the log de Rham fundamental group analogous to the usual de Rham theory \cite[Section~12]{Deligne:groupefondamental}. In this section, all schemes will be over $\C$. To allow for a bit of flexibility, we will attach a fundamental group to a shortened complex of differential forms.

\begin{definition}
    Given a morphism $f\colon X\to Y$ and a shortened complex of differential forms $\tilde{\Omega}$ on $X$ over $Y$, a \emph{vector bundle with $\tilde{\Omega}$-connection} $(\cE,\nabla)$ is a  locally free sheaf $\cE$ on $X$ together with a $f^{-1}\cO_Y$-linear morphism of sheaves 
    \[\nabla\colon \cE\to \tilde{\Omega}^1\otimes \cE\]
    satisfying $\nabla(fs)=f\nabla(s) + df\otimes s$. 

    Given $(\cE,\nabla)$, one may form the operator
    \[\hat{\nabla} \colon \tilde{\Omega}^1\otimes \cE\to \tilde{\Omega}^2\otimes \cE,\ 
    \hat{\nabla}(\omega\otimes s)=d_{\tilde{\Omega}}\omega \otimes s - \omega \wedge \nabla(s).\]
    We say $\nabla$ is integrable if $\hat{\nabla}\circ\nabla=0$.
    The vector space of horizontal sections is
    \[\Gamma(X,(\cE,\nabla))\coloneqq \{s\in\Gamma(X,\cE)\mid \nabla s=0\}.\]
\end{definition}

Given a shortened complex $\tilde{\Omega}$ attached to $f\colon X\to Y$, we can define 
tensor products, duals, and $\Hom$-objects for vector bundles with integrable $\tilde{\Omega}$-connections
as in the usual case. 
The unit object $\mathbf{1}$ is given by the trivial bundle $\cO_X$ equipped with connection form $\nabla=d$. A morphism $f\colon (\cE_1,\nabla_1)\to (\cE_2,\nabla_2)$ is a horizontal vector bundle map, i.e.,~$f\colon \cE_1\to \cE_2$ such that $f(\nabla_1(s))=\nabla_2(f(s))$. There is a canonical isomorphism
\begin{equation} \label{e:classicalhoms}
\Hom(\cE_1,\cE_2)\cong \Gamma(X,(\cE_1,\nabla_1)^\vee\otimes (\cE_2,\nabla_2)).
\end{equation}
We define $\sC(X/Y,\tilde{\Omega})$, {\em the de 
Rham rigid abelian tensor category} attached to $\tilde{\Omega}$ to be that of locally free sheaves $\cE$ equipped with an integrable 
$\tilde{\Omega}$-connection $\nabla$. We write $\sC((X,M_X)/(Y,M_Y))$ when $\tilde{\Omega}$ is defined from a log morphism $f\colon (X,M_X)\to (Y,M_Y)$ and will suppress $(Y,M_Y)$ when it is understood.
We use the superscript ``$\un$'' to denote the corresponding unipotent category.

Objects in the category can be understood more explicitly when the underlying bundle is trivial, in which case the connection can be written as $\nabla=d-\theta$ for 
\[\theta\in\Gamma(\tilde{\Omega}^1\otimes \End(\cE))\]
satisfying $d\theta-\theta\wedge\theta=0$.
When $X$ is proper and both $\cE_1$ and $\cE_2$ are trivial, a morphism $f\colon (\cE_1,\nabla_1)\to (\cE_2,\nabla_2)$ is a constant map on vector bundles such that the following diagram commutes:
\[\xymatrix{\cE_1\ar[r]^f\ar[d]^{\theta_1}& \cE_2\ar[d]^{\theta_2}\\
\tilde{\Omega}^1\otimes \cE_1\ar[r]^{1\otimes f}&\tilde{\Omega}^1\otimes \cE_2}. \]
We can make use of trivializations on components of $X$: an object $(\cE,\nabla)$ of $\sC(X/Y,\tilde{\Omega})$ is {\em component-trivial} if the bundle underlying $\cE|_{X_v}$ is trivial for each irreducible component $X_v$ of $X$ (although the connection may be nontrivial).  
\begin{definition}
    A log scheme $(X,M)$ 
    is {\em locally topologically acyclic} if each stratum $X_F$ is proper and  $H^1(X_F,\cO_{X_F})=0$.
\end{definition}

\begin{lemma}
    If $(X,M)$ is locally topologically acyclic, then each unipotent bundle on $(X,M)$ is component-trivial.
\end{lemma}

\begin{proof}
  We induct on the rank of a unipotent vector bundle $\cE$ on $(X,M)$. If the rank of $\cE$ is greater than $1$, then there is a short exact sequence of underlying vector bundles
  \[\xymatrix{
  0\ar[r]&\cE'\ar[r]&\cE\ar[r]&\cO\ar[r]&0}.\]
  We restrict these bundles to a component $X_v$.
  By induction, $\cE'|_{X_v}$ is a trivial bundle of rank $n-1$. Then, the extension is classified by 
  \[\Ext^1(\cO_{X_v},\cE'|_{X_v})\cong \Ext^1(\cO_{X_v},\cO_{X_v})^{\oplus (n-1)}\cong H^1(X_v,\cO_{X_v})^{\oplus (n-1)}\cong 0.\]
\end{proof}

\begin{definition} 
For a point $x_0\in X(\C)$, the {\em fiber functor at $x_0$} is $\omega_{x_0} \colon\sC^{\un}(X/Y,\tilde{\Omega})\to \Vect_\C$ is given by $\cE\mapsto \cE|_{x_0}$.
For $(X,M)$ locally topologically acyclic, the {\em regular sections fiber functor at a stratum $X_F$} is $\omega_F\colon\sC^{\un}(X/Y,\tilde{\Omega})\to \Vect_\C$ given by 
$\cE\mapsto \Gamma(X_F,\cE|_{X_F})$
where $\Gamma$ denotes regular (not necessarily horizontal) sections.
If $x_0\in X_{F}(\C)$, the {\em evaluation path} is the isomorphism of fiber functors 
$\delta_{Fx_0}\colon\omega_{F}\xrightarrow{\cong}\omega_{x_0}$
given by evaluating sections at $x_0$.
\end{definition}

The fiber functor $\omega_{x_0}$ is faithful. Indeed, since $\Hom((\cE_1,\Delta_1),(\cE_2,\Delta_2))=\Gamma((\cE_1,\nabla_1)^\vee\otimes (\cE_2,\nabla_2))$, it suffices to show $\on{ev}_x\colon\Gamma(\cE,\nabla)\to\cE_x$ is injective for any $(\cE,\nabla)$. This is clearly true for $\mathbf{1}$. We induct on the length of $\cE$. Given an exact sequence
\[\xymatrix{0\ar[r]&(\cE',\nabla')\ar[r]&(\cE,\nabla)\ar[r]&\mathbf{1}\ar[r]&0}\]
we have a commutative diagram whose outer vertical arrows are injective
\[\xymatrix{
0\ar[r]&\Gamma(\cE',\nabla')\ar[r]\ar[d]&\Gamma(\cE,\nabla)\ar[r]
\ar[d]&\Gamma(\mathbf{1})\ar[d]\\
0\ar[r]&\cE'_x\ar[r]&\cE_x\ar[r]
&\mathbf{1}_x.
}\]
Now, the four lemma applies. Since $\omega_{x_0}$ factors through $\omega_F$   for $x_0\in X_F(\C)$, $\omega_F$ is faithful.

\begin{rem} \label{r:restrictedconnection}
    Let $(X,M)/S$ be a smooth variety with log structure induced by a strict normal crossings divisor $D$. Let $(\cE,\nabla)$ be a vector bundle with integrable connection on $(X,M)$. If $D_v$ is a component of $D$, we can consider the following composition where the last map is given by restriction:
    \[\xymatrix{
    \cE\ar[r]^>>>>{\nabla}& \Omega^1_{(X,M)/S}\otimes \cE\ar[r]&  \Omega^1_{(X,M)/S}\otimes \cE\otimes \cO_{D_v}}.\]
    On an open set $U$ for which $D_v$ is cut out by a regular function $f$, 
    \begin{align*}
        \nabla(fs)&=df\otimes s+f\nabla s\\
           &=f\left(\frac{df}{f}\otimes s+\nabla s\right),
    \end{align*}
    and so the composition factors through $\cE\otimes \cO_{D_v}$.
    Similarly, for any stratum $D_F$ of $D$, there is a well-defined restricted connection
    \[\nabla_F \colon \cE\otimes \cO_{D_F}\to (\Omega^1_{(X,M)/S}\otimes \cE)\otimes \cO_{D_F}.\]
    If $(D_F,M_F)$ is the log scheme with inverse image log structure,  one has the
    restriction functor 
    \[\sC((X,M)/S)\to \sC((X_F,M_F)/S),\ (\cE,\nabla)\to (\cE\otimes \cO_{X_D},\nabla_F).\]
    The analogous statement is true for strata of strict normal crossings schemes.
    For $x_0\in X_F(\C)$, $\omega_{x_0}$ factors through the restriction functor, so these restriction functors are faithful.
\end{rem}

\subsection{Kato--Nakayama spaces}
When the log structure on $(X,M)$ is fine saturated, the underlying scheme of $X$ is of finite type, and the morphism $(X,M)\to S^\dagger$ is log smooth (which holds for simple normal crossings schemes), the Tannakian category $\cC^{\un}((X,M)/S^\dagger)$ and its corresponding fundamental group can be related to the unipotent completion of the classical fundamental group of the Kato--Nakayama space \cite{KatoNakayama} of $(X,M)$. 
For the log analytic space attached to $(X,M)$, the {\em Kato--Nakayama space} $X^{\log}$ is defined to be a certain topological space equipped with a sheaf $\cO_X^{\log}$. If $(X,M)$ arises as the closed fiber of a semistable degeneration, $X^{\log}$ is a topological stand-in for the generic fiber of the family. When the log structure is attached to a normal crossings divisor, $X^{\log}$, is a ``real blow-up'' at the divisor. We recommend \cite{Ogus:logbook} for more details. 
    
The Kato--Nakayama space of $S^\dagger$ is $S^1$ and $\cO_{S^\dagger}^{\log}$ is the the symmetric algebra on the pushforward of the locally constant sheaf $\underline{\R}$ under the exponential map $\R\to S^1$, that is, polynomials in the angular coordinate $\theta$.  Morphisms of log schemes $(X,M)\to S^{\dagger}$ induce continuous maps $X^{\log}\to S^1$. We write $X^{\log}_{\zeta}$ for the fiber over some $\zeta\in S^1$.

For $(X,M)/S^\dagger$, a point $x\in X(\C)$ is {\em relatively trivial} if the relative characteristic sheaf $\overline{M}_{X/S^\dagger}$ vanishes at $x$. In the case of a strict normal crossing scheme with divisor at infinity over $S^\dagger$, these would be the $\C$-points in a maximal open stratum. For such a point, there is a $S^\dagger$-morphism $S^\dagger\to (X,M)$ with $x$ as its image inducing a map $S^1\to X^{\log}$. We write $(x^{\log},\zeta)$ for the image of $\zeta\in S^1$ under this map, and $\omega_{(x^{\log},\zeta)}$ for the corresponding fiber functor.

For a topological space $Z$, let $\sC^B(Z)$ denote the Tannakian category of $\C$-local systems on $Z$.

\begin{theorem} \label{t:kniso}
    Let $\zeta\in S^1$.
    There is an equivalence of categories
    \[G\colon \sC^{\un}((X,M)/S^\dagger) \to \sC^{B,\un}(X_\zeta^{\log})\]
    where for a relatively trivial point $x\in X(\C)$, the fiber functor $\omega_x$ is taken to $\omega_{(x^{\log},\zeta)}$.
\end{theorem}

The proof is a slight extension of Theorem~(0.5) of \cite{KatoNakayama} which is stated for schemes over a trivial log point and which  works with the category of vector bundles with integrable connections with nilpotent residues. One observes that the functor $G$ is exact and takes line bundle with trivial connection to a trivial one-dimensional local system, and thus preserves unipotent objects. To generalize the theorem to log schemes over $S^\dagger$, one uses the relative log Poincar\'{e} lemma of \cite{FKato:relativePL} and observes an isomorphism between the log de Rham cohomology of $(X,M)$ and the Betti cohomology of $X^{\log}_\zeta$ as per \cite[Section~4]{FKato:relativePL}. Otherwise, the proof carries through verbatim.
The statement about fiber functor follows from the definition of $G$.

\begin{corollary} \label{c:kniso}
    In the above situation, there is an isomorphism of Tannakian fundamental groups
    \[\pi_1^{\un}((X,M),x)\to \pi_1^{B,\un}(X_\zeta^{\log},(x^{\log},\zeta)).\]
\end{corollary}

We will need the following result connecting Kato--Nakayama spaces to the fibers of a family:
\begin{theorem} \label{t:knhomotopy}
    Let $U\subset \C$ be an open set containing the origin. Let $\cX\to U$ be a proper analytic semistable family where $X_0$ is the fiber over $0$. Let $\cD\subset \cX$ be a divisor intersecting the strata of $X_0$ transversely. Give $\cX$ the log structure induced by the divisor $X_0\cup \cD$ and let $(X_0,M_{X_0})$ be the inverse image log structure so that it is that of a strict normal crossings scheme with divisor at infinity induced by $\cD\cap X_0$. Then, for any $t\in U\setminus\{0\}$ arbitrary close to $0$, there is  a homotopy equivalence between $(X_0^{\log})_\zeta$ and $\cX_t\setminus \cD_t$.

    Given a $C^1$-path $\delta$ between $0$ and $t$ in $\{s\in U\mid \|s\|\leq \|t\|\}$ with $\arg(\delta'(0))=\zeta$ and its lift $\tilde{\delta}$  in $\cX$ between a point $x_0\in X_0(\C)$ with relatively trivial log structure 
    and some $x_t\in X_t(\C)$, there is an isomorphism
    \[\pi_1((X_0^{\log})_\zeta,(x_0)_{\zeta})\cong \pi_1(\cX_t\setminus \cD_t,x_t)\]
\end{theorem}

\begin{proof}
  By passing to an open subset of $U$ and rescaling, we may suppose $U$ is an open disc and that 
  $\cX$ is a submersion over $U^*=U\setminus\{0\}$. Give $U$ the log structure of a standard 
  log disc (i.e.,~with log structure corresponding to the origin as a divisor). Then 
  $(\cX,M_{\cX})\to (U,M_U)$ satisfies the hypotheses of \cite[Proposition~5.1]{NakayamaOgus}, 
  and $\cX^{\log}\to U^{\log}$ is a fiber bundle. Consequently, there is a homotopy 
  equivalence between the fiber over $t$ and the fiber over $\zeta\in (\{0\})^{\log}\subset U^{\log}$, which 
  is exactly $(\cX_0^{\log})_\zeta$. Moreover, it is a standard result that 
  $X_t\setminus\cD_t\hookrightarrow X^{\log}_t$ is a homotopy equivalence. 

  We can lift $\tilde{\delta}$ to a path in $\cX^{\log}$ between $(x_0)_{\zeta}$ and $x_t$. This induces the isomorphism of fundamental groups.
\end{proof}

\subsection{Descent for the log unipotent de Rham category}

We will prove that the unipotent category of vector bundles with connection on a strict normal crossings scheme with a divisor at infinity is equivalent to the unipotent subcategory of a suitable descent category. To do so, we must use descent results from Appendix~\ref{a:appendix}. For general background on descent, see \cite{Vistolidescent} or \cite[Chapter~6]{BLR}. 

Let $(X_0,M)$ be a log scheme over $S^\dagger$ with the structure of a simple normal crossings scheme with divisor at infinity. Let $\Gamma$ be the dual complex of $(X_0,M)$, i.e.,~the simplicial complex whose vertices are given by the components of $X_0$ and each face corresponds to the intersection of corresponding components. For a face $F$ of $\Gamma$, let $X_F$ denote the corresponding closed stratum. Define the log $S^\dagger$-scheme $(X_F,M_F)$ by letting $M_F$ be the inverse image log structure induced by $X_F\hookrightarrow X_0$.
Observe that for an inclusion of faces $F_1\subset F_2$, there is a morphism of log $S^{\dagger}$-schemes 
 $(X_{F_2},M_{F_2})\to (X_{F_1},M_{F_1})$.

\begin{definition}
Given a strict normal crossings scheme with divisor at infinity $(X_0,M)$ such that each $X_F$ is connected, we define the Tannakian category $\sC_{\Gamma}$ over $\Gamma$ of $(X_0,M)$ as follows:
\begin{enumerate}
    \item for each face $F$ of $\Gamma$, set $\sC_F=\sC^{\un}((X_F,M_F)/S^\dagger)$, and
    \item for an inclusion of faces $F_1\subset F_2$, let $i_{F_1F_2}\colon \sC_{F_1}\to \sC_{F_2}$ be given by restriction of connections (as in Remark~\ref{r:restrictedconnection}) induced by $(X_{F_2},M_{F_2})\to (X_{F_1},M_{F_1})$.
\end{enumerate} 
  Let $\sC(\Gamma)$ be the descent category.
\end{definition}

\begin{rem}
    The stratum $(X_F,M_F)$ is considered as a log $S^\dagger$-scheme. In the case where $X_0$ arises as the closed fiber of the closure of a sch\"{o}n subvariety $X$ in a toric scheme $\cP(\Sigma)$ over $\cO$ as in Proposition~\ref{p:schonlog}, by Lemma~\ref{l:logstructureontransversetoorbits}, its shortened complex isomorphic to the pullback by $X_F\hookrightarrow\PP(\Sigma_F)$  where $\PP(\Sigma_F)$ is given the shortened complex induced by the enriched fan with partial compactification $\Sigma_F$ over $S$.
\end{rem}

We can pick a fiber functor $\omega$ on $\cC^{\un}((X_F,M_F)/S^\dagger)$ for some face $F$ to obtain a fiber functor on $\sC(\Sigma)$.

\begin{prop} \label{p:derhamdescent}
For a strict normal crossings scheme with divisor at infinity $(X_0,M)$, there is an equivalence of categories between $\sC^{\un}((X_0,M)/S^\dagger)$ and $\sC(\Gamma)^{\un}$.
\end{prop}

\begin{proof}
  Observe first that the log de Rham complex $\Omega^\bullet_{(X_0,M_X)/S^\dagger}$, considered as a complex of vector bundles restricts to each $F$ as $\Omega^\bullet_{(X_F,M_F)/S^\dagger}$ where the following diagram commutes for $F_1\subset F_2$
  \[\xymatrix{
  \Omega^i_{(X_0,M)/S^\dagger}|_{X_{F_2}}\ar[r]\ar[dr]&\Omega^i_{(X_{F_1},M_{F_1})/S^\dagger}|_{X_{F_2}}\ar[d]\\
  &\Omega^i_{(X_{F_2},M_{F_2})/S^\dagger}.
  }
  \]
  Moreover, any vector bundle $\cE$ on $X_0$ induces a collection $\{\cE_F\}$ of vector bundles related by $i_{F_1F_2}$. By restriction, the connection $\nabla$ induces a $\Omega_{(X_F,M_F)/S^{\dagger}}$-connection $\nabla_F$ on $\cE_F$. This defines the functor $\sC(X,M)\to \sC(\Sigma)$. We will show that this functor is fully faithful and essentially surjective, thus an equivalence of categories.

  To prove essential surjectivity, note that Theorem~\ref{t:amainresult} allows us to descend $\{\cE_F\}$ to a vector bundle $\cE$ on $X_0$. 
  Now, we explain how to glue connections over strata.
  Given a local section $s$ of $\cE$ whose domain intersects $X_v$, 
  consider $\nabla_v(s|_{X_v})$ as a local section of $\Omega^1_{(X_v,M|_{X_v})/S^\dagger}\otimes \cE|_{X_v}$. Because $\nabla_v$ form a system of connections whose restrictions are related by $i_{v,e}$, the collection $\{\nabla_v(s|_{X_v})\}$ glues to a section $\nabla s$ of $\Omega^1_{(X,M)/S^\dagger}\otimes \cE$.  Because, for a local regular function $f$, the collection $\{df|_{X_v}\}$ of local sections of $\Omega^1_{(X_v,M_{X_v})/S^\dagger}$  glues to $df$, $\nabla$ obeys the Leibniz rule. The connection is integrable because each $\nabla_v$ is. Thus, we obtain $(\cE,\nabla)$.

  To prove that the functor is fully faithful, note that $\Hom((\cE_1,\nabla_1),(\cE_2,\nabla_2))$ can be identified with horizontal sections of $(\cE_1,\nabla_1)^\vee\otimes (\cE_2,\nabla_2)$. These horizontal sections uniquely glue from those of 
  $\{(\cE_{1,F},\nabla_{1,F})^\vee\otimes (\cE_{2,F},\nabla_{2,F})\}$.
\end{proof}

We can specialize Proposition~\ref{p:derhamdescent} to the case of sch\"{o}n subvarieties, following Proposition~\ref{p:schonlog} in the special case that each stratum is connected.

\begin{prop} \label{p:schondescent}
Let $X\subset(K^*)^n$ be a sch\"{o}n variety. 
Pick $\Sigma$ as in Proposition~\ref{p:dilatesnc}, and let $\cX$ be the closure of $X$ in $\cP(\Sigma)_{\cO'}$ for $\cO'=\cO[\pi^{\frac{1}{d}}]$. Let $R$ be a set of rays of $\Sigma_0$ of $\Sigma$, and let $\cP(\Sigma)_{\cO}$ be given the log structure induced from $(\Sigma,N\coloneqq\Z^n,R)$ over $S^\dagger$.
Let $(X_0,M_0)$ be the closed fiber of $\cX$ with inverse image log structure.

 For $F\in \Sigma$, consider $(X_F,M_F)$ where $X_F=\cX\cap \PP_F$ equipped with inverse image log structure from $X_F\hookrightarrow \cP(\Sigma)_0$. Let $\Gamma$ be the union of the bounded cells of $\Trop(X)$. Suppose that  $X_F$ is connected for each $F\in\Gamma$.

Set $\sC_F\coloneqq \sC^{\un}((X_F,M_F)/S^\dagger)$ and for an inclusion of faces $F_1\subset F_2$, let the functor $\sC_{F_1}\to \sC_{F_2}$ be induced by pullback $X_{F_2}\hookrightarrow X_{F_1}$. Then, there is an equivalence of categories between $\sC^{\un}((X_0,M)/S^\dagger)$ and $\sC(\Gamma)^{\un}$.
\end{prop}

\begin{proof}
    The complex $\Gamma$ is the dual complex of $X_0$ by the discussion in \cite[Section~4]{HelmKatz} where it is the bounded part of the parameterizing complex.
    Because $(X_0,M)$ is a strict normal crossings scheme with divisor at infinity, there is an equivalence of categories between $\sC^{\un}((X_0,M_0)/S^\dagger)$ and $\sC(\Gamma)^{\un}$ as in Proposition~\ref{p:derhamdescent}. 
\end{proof}

\section{Correspondence Theorem} \label{s:correspondencetheorem}

\subsection{Reduction to log de Rham/tropical comparison}

In this subsection, we prove our Theorem~\ref{t:maintheoremintro} as a consequence of the following  technical result to be proven in later subsections:
\begin{prop} \label{p:mainmaintheorem}
Let $N=\Z^n$. Let $X\subseteq (K^*)^n$ be a tropically smooth subvariety, and let $\Sigma$ be a complete unimodular rational polyhedral complex supporting $\Trop(X)\subseteq N_{\R}$. Give $\Trop(X)$ the polyhedral complex 
structure induced by $\Sigma$.
Let $R$ be a collection of rational rays contained in the recession fan $\Sigma_0$.

Let $X_0$ be the closed fiber of the closure $\cX\coloneqq\overline{X}$ in $\cP(\Sigma)_{\cO}$. Suppose that $X_0$ is a transverse-to-strata subscheme of $\cP(\Sigma)_0$. 
 Give $\cP(\Sigma)_{\cO}$ the log structure induced by $(\Sigma,N,R)$, and  let $(X_0,M)$ be the inverse image log structure (over $S^\dagger$) induced by inclusion.

Then, there is an equivalence of categories
\[\sC^{\un}(\Trop(X),N,R)\cong \sC^{\un}((X_0,M)/S^\dagger).\]

For a bounded polyhedron of $P\in\Sigma$, let $\omega_{X_0,P}$ and $\omega_{\Trop(X),P}$ be the fiber functors in the respective categories. Then, there is an isomorphism of fundamental groups
\[\pi_1^{\un}((X_0,M)/S^\dagger,\omega_{X_0,P})\cong \pi_1^{\un}((\Trop(X),N,R),\omega_{\Trop(X),P})\]
\end{prop}

\begin{proof}[Proof of Theorem~\ref{t:maintheoremintro}]  
  We apply Proposition~\ref{p:dilatesnc} to produce a fan $\Sigma$ whose recession fan $\Sigma_0$ is a subfan of a refinement of $\Delta$. Because $X$ is defined over the function field of the ring of germs of holomorphic functions near the origin in $\C$, it corresponds to a family of subvarieties $\{X_t\}$ in $(\C^*)^n$ over a punctured open disc $U^*\subset\C$. 
  By sch\"{o}nness, after possibly rescaling and passing to a branched cover over the origin (corresponding to the basechange $\cO[t^{\frac{1}{d
  }}]\to\cO$), we may suppose that the closure $\overline{X}$ of $X$ in $U^*\times\PP(\Sigma_0)$ is smooth and submersive over $U^*$. Give $\cX$, the closure of $X$ in $\cP(\Sigma)_{\cO}$, the log structure attached to the strict normal crossings divisor $X_0\cup \cD$ where $X_0$ is the fiber over $0$ and 
  \[\cD\coloneqq \cX\cap\left(\bigcup_{\rho\in(\Sigma_0)_{(1)}\setminus R} D_\rho\right).\]
  Write $\mathring{X}'_t$ for the fiber of $\cX\setminus \cD$ over $t$. Note that $p\colon \PP(\Sigma_0)\to \PP(\Delta)$ induces a birational morphism 
  from $\mathring{X}'_t$ to $p(\mathring{X}'_t)=\mathring{X}_t$. 
  Hence, 
  \[p\colon \pi_1(\mathring{X}'_t,x_t)\to \pi_1(\mathring{X}_t,x_t)\]
  is an isomorphism, and 
  we may work with $\mathring{X}_t$ replaced by $\mathring{X}'_t$.
  The latter space is homotopy equivalent to ${X_{0\zeta}^{\log}}$ for $\zeta=\arg(t)$ 
  by Theorem~\ref{t:knhomotopy}. Hence, by 
  Corollary~\ref{c:kniso}, 
  \[\pi_1^{\un}(\mathring{X}_t,x_t)\cong \pi_1^{\un}((X_0,M)/S^{\dagger},x_0)\]
  for a point $x_0\in X_0(\C)$ with relatively trivial log structure. Since $x_0$ lies on $X_v$ for some $v\in\Sigma_{(0)}$, there is a canonical evaluation path between $\omega_{X_0,v}$ and the fiber functor attached to $x_0$. By applying Proposition~\ref{p:mainmaintheorem}, we, thus, obtain the isomorphism
  \[\pi_1^{\un}((X_0,M)/S^{\dagger},\omega_{X_0,v})\cong \pi_1^{\un}((\Trop(X),\Z^n,R),\omega_{\Trop(X),\omega_v}).\]
  Since the tropical fundamental group is unaffected by refinement, we have the desired conclusion.
\end{proof}


\subsection{Residues and Hyperplane Arrangements}

We first review some properties of residues of differentials with log poles (for more details, see e.g.,~\cite[Section~4.2]{Peters-Steenbrink}).

\begin{definition} Let $X$ be a smooth variety with strict normal crossings divisor $D=D_1\cup \dots D_N$. For a subset $S\subset\{1,\dots,N\}$, write 
\[D_S=\cap_{i\in S} D_i,\ 
D(S)=D_S\cap \left(\bigcup_{j\not\in S} D_j\right).\]
Here, $D(S)$ is the divisor induced by the other components of $D$ on $D_S$.
For a log $p$-form $\omega\in\Gamma(X,\Omega^p(\log D))$ and and ordered tuple $I=(i_1,\dots,i_k)$ with distinct indices $i_j$, if $D_I\neq \emptyset$,
we will define the residue on $D_I$ (considered as a stratum together with the ordering $I$ of the divisors containing it), 
\[\Res_{D_I}(\omega)\in\Gamma(D_I,\Omega^{p-k}(\log D(I))).\]
Near a point of $D_I$, pick local coordinates such that $D_{i_j}$ is given by $z_j=0$.
Decompose
\[\omega=\frac{dz_1}{z_1}\wedge \dots\wedge \frac{dz_k}{z_k}\wedge \eta+\omega'\]
where $\omega'$ does not involve all of $\left\{\frac{dz_1}{z_1},\dots,\frac{dz_k}{z_k}\right\}$
in a single wedge product.
Set $\Res_{D_I}(\omega)=\eta|_{D_I}.$
The residues globalize to a well-defined log $(p-k)$-form on $D_I$ that 
is alternating in permutations of $I$.
\end{definition}

The following is a verification in local coordinates:
\begin{prop} \label{p:residueproperties}
Let $X$ be a smooth proper variety with strict normal crossings divisor $D$. Let $\omega\in\Gamma(X,\Omega^p(\log D))$ be a log $p$-form. Then,
\begin{enumerate}
    \item \label{i:transverseresidues} for $Z\subset X$, a transverse-to-strata subvariety, for any $I$ such that $D_I\cap Z\neq \varnothing$,
       \[\Res_{D_I\cap Z}(\omega|_Z)=\Res_{D_I}(\omega)|_{D_I\cap Z}.\]
    \item \label{i:iteratedresidue} if $I,J$ are supported on disjoint subsets of $\{1,\dots,N\}$, then
    \[\Res_{D_J\cap D_I}(\Res_{D_I}(\omega))=\Res_{D_{IJ}}(\omega)\]
\end{enumerate}

In the above, we write $IJ$ for concatenation of sequences.
\end{prop}

\begin{rem} \label{r:uniqueness1form} For a smooth complete toric variety $\PP(\Delta)$ equipped with the standard toric log structure, the vector space of global log $1$-forms is generated by $\DLog(m)$ for $m\in M$, and hence can be identified with $M_\C$. Similarly, global log $p$-forms can be identified with $\bigwedge^p M_{\C}$. To a ray $\rho\in\Delta$, there is an associated divisor $D_\rho$. Write $u_\rho\in M$ for the primitive lattice point on $\rho$. Then, given $I=(\rho_1,\dots,\rho_k)$ and $p$-form $\omega\in \bigwedge^p M_{\C}$, the residue $\Res_{D_I}(\omega)$ is the iterated interior product
 \[\iota_{u_{\rho_k}}(\iota_{\rho_{k-1}}\dots (\iota_{u_{\rho_1}} \omega)))\in\bigwedge\nolimits^{p-k} (M(\Span_{\geq 0}(\rho_1,\dots,\rho_k)))_\C\]
 where $M(\Span_{\geq 0}(\rho_1,\dots,\rho_k)))=(\Span(\rho_1,\dots,\rho_k))^\perp.$
In particular, for a $1$-form $\DLog(m)$ and a ray $\rho$ with primitive integer vector $u_\rho$,
$\Res_{D_\rho} (\DLog(m)) = \langle m,u_{\rho}\rangle.$
\end{rem}

\subsection{Subspace  compactifications}

We will need to study log forms on compactified hyperplane arrangements complements.  

\begin{definition}
    Let $\PP(\Delta)$ be a proper smooth toric variety with open torus $T$. A transverse-to-strata subvariety $X\subset\PP(\Delta)$ is a {\em subspace compactification} if $X=\overline{X\cap T}$, and there is a $T$-equivariant compactification of $T$ to the toric variety $\PP(\Delta_n)=\PP^n$ such that the closure of $X\cap T$ in $\PP^n$ is a projective subspace $P$.
\end{definition}

We can view $P\cap T$ as the complement in $P$ of the coordinate hyperplanes of $\PP^n$, thus it is a hyperplane arrangement complement.
Give $X$ the inverse image log structure induced from the closed embedding in $\PP(\Delta)$. Write $D$ for the the strict normal crossings divisor on $X$ arising from the intersection with the toric divisors of $\PP(\Delta)$, so that log $1$-forms on $X$ are sections of $\Omega^1(\log D)$. We will prove that the restriction of these log $1$-forms to $X\cap T$ are pulled back from log $1$-forms on $\PP^n$.

\begin{definition}
Let $X\subset \PP(\Delta)$ be a transverse-to-strata subvariety of a proper smooth toric variety.
The {\em tropical span} $N_{X}\subset N_\C$ of $X$ is the $\C$-linear span in $N_\C$ of the rays $\rho$ of $\Delta$ for which $X\cap V(\rho)\neq \varnothing$. The inclusion $N_X\hookrightarrow N_{\C}$ is dual to a projection $M_{\C}\to M_{X}\coloneqq N_{X}^\vee$.     
\end{definition}

\begin{lemma} \label{l:1formispullback}
Let $X\subseteq \PP(\Delta)$ be a subspace compactification. Let $D$ be the union of the divisors of $X$ arising from intersection with the toric divisors of $\PP(\Delta).$
The restriction of any log $1$-form $\omega\in \Gamma(X,\Omega^1_{X}(\log D))$ to $X\cap T$ is the pullback of a log $1$-form on $\PP(\Delta)$, which is unique as an element of $M_X$.
\end{lemma}

\begin{proof}
Without loss of generality, we may suppose that $\Delta$ is a refinement of $\Delta_n$. 
  Write $p\colon\PP(\Delta)\to \PP(\Delta_n)$ for the birational morphism and $P=\overline{p(X)\cap T}\subset \PP^n$.
Pick homogeneous coordinates $\{X_0,\dots,X_n\}$ on $\PP(\Delta_n)=\PP^n$. Write $H_i$ for the hyperplane $\{X_i=0\}$.
Let \[I_0\sqcup\dots\sqcup I_k=\{0,\dots,n\}\]
be a partition such that $i,i'\in I_j$ for some $j$ if and only if $P\cap H_i=P\cap H_{i'}$. These are the rank $1$ flats of the matroid attached to $P$. Pick $i_j\in I_j$ for all $j$ such that (after possible reordering) $i_0=0$. Let $D_j\subset X$ be the proper transform of the divisor $P\cap H_{i_j}$ under $p\colon X\to P$. 

We claim that for an toric divisor $E$ of $\PP(\Delta)$, either $X\cap E=D_j$ for some $j$ or $p(X\cap E)$ is of codimension at least $2$ in $P$. Because $E$ is a toric divisor, $p(E)\subseteq H_i$ for some $i$. Hence $p(X\cap E)\subseteq P\cap H_i$. If this is an equality, then $X\cap E=D_j$ for some $j$. Otherwise, $p(X\cap E)$ is of codimension at least $2$. 

Consider log $1$-form on $\PP(\Delta)$
\[\omega'\coloneqq \omega-\sum_{j=1}^k \Res_{D_j}(\omega) p^*\left(\frac{dX_{i_j}}{X_{i_j}}-\frac{dX_0}{X_0}\right)\Big|_X.\]
Let $L$ be a generic line in $P$ and $\tilde{L}$ be its proper transform in $X$. Because $L$ only intersects the divisors $p(E)$ for which $P\cap p(E)=P\cap H_i$, the only toric divisors that $\tilde{L}$ intersects are the $D_j$'s. Then, by construction, $\omega'|_{\tilde{L}}$ only has poles along the point $\tilde{L}\cap D_0$. By the residue theorem, this residue must also vanish. Hence $\omega'$ is regular along all the $D_j$'s. 

Let $U$ be the maximal open subset of $P$ for which $p\colon p^{-1}(U)\to U$ is an isomorphism. Since $\omega'|_U$ is regular, it extends to $P$ as a regular $1$-form by a Hartogs's phenomenon argument (see, e.g.,~\cite[Theorem II.8.19]{Hartshorne}). Hence, $\omega'=0$. 

The above argument shows that the $1$-form $\omega$ is determined by its residues on $X\cap E$ for toric divisors $E$. Hence, it is unique as an element of $M_X$.
\end{proof}

\subsection{Tropical Pullback Map}

Let $(\Delta,\pi\colon N\to N',R)$ be an enriched fan with partial compactification such that $\Delta$ is complete and unimodular. Considered as an enriched fan over $\Delta_S$, it induces a shortened complex on $\PP(\Delta)$.
Then, for $X$  a closed transverse-to-strata subvariety of $\PP(\Delta)$ defined over $\C$, let $\tilde{\Omega}$ be the pullback shortened complex on $X$.
One should have in mind is the following case: $\PP(\Delta)$ is a closed orbit in some larger smooth toric variety (which has the log structure induced by the union of some toric divisors)  and is equipped with the inverse image log structure; $X$ is embedded in $\PP(\Delta)$ and is given the inverse image log structure from $\PP(\Delta)$. 

Consider the fan $\Trop(X^\circ)$ for $X^\circ=X\cap T$ where $T$ is the open torus in $\PP(\Delta)$. Since $X$ is transverse-to-strata, $\Trop(X^\circ)$ is supported on $\Delta$ and can be given the polyhedral structure from it. After placing further hypotheses on $X$, we will define the tropical pullback functor
\[F\colon \sC^{\un}(\Trop(X^\circ),\pi\colon  N\to N',R)\to \sC^{\un}(X/S,\tilde{\Omega}).\]
For convenience, given a cone $\sigma\in\Delta$, write $(X_\sigma,M_\sigma)$ for the stratum $X\cap V(\sigma)$ with inverse image log structure.

\begin{definition}
With notation as above, given a tropical $1$-form 
\[\eta\in\Gamma(\Delta,\Omega^1_{(\Delta,N\to N',R)/\Delta_S})\subseteq M_{\C}/(\Span_{\R}(\Delta))^{\perp_{M'}}_{\C}\]
pick $m\in M_{\C}$ restricting to $\eta$. The {\em tropical pullback} of $\eta$ is 
\[\omega=\DLog(m)\in \Gamma(X,\tilde{\Omega}^1).\] 
\end{definition}

Because the closed embedding $X\cap (\C^*)^n\hookrightarrow(\C^*)^n$ factors through the subtorus of $(\C^*)^n$ corresponding to $\Span_{\R}(\Delta)$, the log $1$-form $\omega$ is independent of the choice of $m$.

Given an integrable tropical connection 
\[\eta\in\Gamma(\Delta,\Omega^1_{(\Delta,N\to N',R)/\Delta_S}\otimes\End(E))=\Gamma(\Delta,\Omega^1_{(\Delta,N\to N',R)/\Delta_S})\otimes\End(E(\Trop(X^\circ)))\] 
on a trivial bundle $E$ on $\Delta$ (i.e.,~one satisfying $\eta\wedge\eta=0$), we would like to define $\theta$ as the tropical pullback of $\eta$, and define the connection $\nabla=d-\theta$. 
To ensure that $\nabla$ is integrable, we will have to impose some additional conditions on $X$ as follows:

\begin{definition} \label{d:skow}
    A closed subvariety of a smooth toric variety $X\subseteq \PP(\Delta)$ is {\em some-kind-of-wonderful} if 
    \begin{enumerate}
     \item $X$ is smooth and transverse-to-strata in $\PP(\Delta)$,
     \item $H^0(X,\Omega^1_X)=H^0(X,\Omega^2_X)=0$, and
     \item \label{i:skow3} for each toric divisor $D_\rho$ of $\PP(\Delta)$, $H^0(X_\rho,\Omega^1_{X_\rho})=0$ where $X_\rho=X\cap D_\rho$.
\end{enumerate}
\end{definition}

Subspace compactifications are some-kind-of-wonderful. Indeed, they have no regular differential forms. Moreover, the star-quotient of every face of their tropicalizations is a Bergman fan by Lemma~\ref{l:anystructure}, and so by applying  
Lemma~\ref{p:matroidtrop} to $X_\rho$, we see $X_\rho$ has no non-trivial regular differential forms.

\begin{prop} \label{p:troppullback}
  Let  $X\subset\PP(\Delta)$ be a some-kind-of-wonderful subvariety of $\PP(\Delta)$ over $\C$. Consider the enriched fan with partial compactification $(\Delta,N\to N',R)$. Give $\PP(\Delta)$ the shortened complex induced from $(\Delta,N\to N',R)$, and pull it back (Definition~\ref{pullbackscodf}) to $X$ to produce the shortened complex $\tilde{\Omega}$.
 
  Give $\Trop(X^\circ)$ the polyhedral complex structure induced from $\Delta$, and let $(E,\eta)$ is a rank $n$ tropical vector bundle with integrable connection on $(\Trop(X^\circ),N\to N',R)$. Let $\cE$ be the rank $n$ trivial bundle on $X$ attached to the vector space $E(\Trop(X^\circ))$.
  
  Then, the tropical pullback $\theta$ of $\eta$ induces an integrable $\tilde{\Omega}$-connection on $\cE$. Moreover, if $\sigma_1\subset \sigma_2$ is an inclusion of cones of $\Delta$ contained in $\Sigma$ such that the strata $X_{\sigma_1}$ and $X_{\sigma_2}$ are some-kind-of-wonderful, then the following diagram commutes:
  \[\xymatrix{
  \on{Ob} \sC^{\un}((\Trop(X^\circ),N\to N',R)_{\sigma_1})\ar[r]\ar[d]&\on{Ob} \sC^{\un}(X_{\sigma_1},\tilde{\Omega}_{\sigma_1})\ar[d]\\
  \on{Ob} \sC^{\un}((\Trop(X^\circ),N\to N',R)_{\sigma_2})\ar[r]&\on{Ob} \sC^{\un}(X_{\sigma_2},\tilde{\Omega}_{\sigma_2})
  }\]
  where $\tilde{\Omega}_{\sigma}$ is the shortened complex on $X_\sigma$ induced from the enriched fan with partial compactification  on the star-quotient $(\Delta,N\to N',R)_\sigma$.
\end{prop}

\begin{proof}
  Note that $\theta$ induces a $\tilde{\Omega}$-connection, $\nabla=d-\theta$ on $\cE$.
  It suffices to prove that $\nabla$ is integrable, i.e.,~that $d\theta-\theta\wedge\theta=0$. Without loss of generality, we can consider the case $R=\varnothing$ where the sheaf of log $1$-forms on $\PP(\Delta)$ is given by $M\otimes\cO_{\PP(\Delta)}$. We can decompose the lattice $M\cong M'\oplus M''$ for some $M''$ to write the sheaf of log $1$-forms as
  \[\Omega^1_{(\PP(\Delta),M_{\PP(\Delta)})/S}\cong \Omega^1_{(\PP(\Delta),M_s)/S}\oplus (M''\otimes \cO_{\PP(\Delta)})\]
  where $(\PP(\Delta),M_s)$ has the standard toric log structure. We will write $(X,M_s)$ for the induced inverse image log structure on $X$
  Therefore, we may decompose $\theta\in\Gamma(X,\Omega_{(X,M)}^1\otimes\End(E))$
  as 
  \[\theta=\theta_s+\theta_h\]
  where 
  \[\theta_s\in\Gamma(X,\Omega_{(X,M_s)}^1\otimes\End(E)),\ 
    \theta_h\in\Gamma(X,(M''\otimes\cO_X)\otimes\End(E)).\]
  We may similarly decompose $\eta=\eta_s+\eta_h$ for
  \[\eta_s\in 
  \Gamma(\Sigma,\Omega^1_{(\Trop(X^\circ),\operatorname{Id}\colon N'\to N',\varnothing)})\otimes \End(E(\Trop(X^\circ))),\ \eta_h\in M''\otimes  \End(E(\Trop(X^\circ))).\]
  Since $d\theta=0$ by construction, we need only show $\theta\wedge\theta=0$. We can write
  \[\theta\wedge\theta=\theta_s\wedge \theta_s+\omega+\theta_h\wedge\theta_h\]
  for 
  \begin{enumerate}
      \item \label{i:i1} $\theta_s\wedge\theta_s\in \Gamma(X,\Omega_{(X,M_s)}^2\otimes\End(E))$,
      \item \label{i:i2} $\omega\in \Gamma(X,\Omega_{(X,M_s)}^1\otimes (M'\otimes \cO_X)\otimes\End(E))$, and
      \item \label{i:i3} $\theta_h\wedge\theta_h\in\Gamma(X,(\bigwedge^2 M''\otimes\cO_X)\otimes\End(E))$.
  \end{enumerate}
  Now, $(X,M_s)$ has the log structure induced by the strict normal crossings divisor $D$ given by the intersection of $X$ with the union of the toric divisors in $\PP(\Delta)$. From $\eta_s\wedge\eta_s=0$, we obtain that all residues of $\theta_s\wedge\theta_s$ on codimension $2$ strata vanish. Hence, from (\ref{i:iteratedresidue}) of Property~\ref{p:residueproperties} for any ray $\rho$ of $\PP(\Delta)$, $\Res_{X_{\rho}}(\theta_s\wedge\theta_s)$ is regular on $X_\rho$. Therefore, $\Res_{X_{\rho}}(\theta_s\wedge\theta_s)$ vanishes by Definition~\ref{d:skow}(\ref{i:skow3}), and so $\theta_s\wedge\theta_s$ is  regular on $X$. Hence, $\theta_s\wedge\theta_s=0$. 

  Similarly, from $\eta_h\wedge \eta_s+\eta_s\wedge \eta_h=0$, we obtain that for any $v\in (M'')^\vee$, the interior product $\iota_v(\omega)$ is a regular $1$-form on $X$ and hence is equal to $0$ by (\ref{i:i2}). Lastly, $\eta_h\wedge\eta_h=0$ implies $\theta_h\wedge \theta_h=0$ by definition.  

  The commutativity of the above diagram follows from definitions.
\end{proof}

We wish to show that the above association $(E,\eta)\mapsto (\cE,\theta)$ induces a functor
\[F\colon\sC^{\trop}(\Trop(X^\circ),\pi\colon N\to N',R)\to \sC^{\un}(X/S,\tilde{\Omega}))\]
of neutral Tannakian categories.
It is straightforward that $F$ respects tensor product, duality, and unit objects.
To define the functor on morphisms, by \eqref{e:tropicalhoms} and \eqref{e:classicalhoms}, it suffices to define a homomorphism
  \[\Gamma(\Trop(X^\circ),(E_1,\eta_1)^\vee\otimes (E_2,\eta_2))\to \Gamma(X,F(E_1,\eta_1)^\vee\otimes F(E_2,\eta_2))\]
  which is a consequence of the following which also demonstrates that $F$ is fully faithful:
  
  \begin{lemma} \label{l:fullyfaithful} Let $(\cE,\nabla=d-\theta)$ be the tropical pullback of $(E,\eta)$ as above. Then,
  there is a canonical isomorphism
  \[\Gamma(\Trop(X^\circ),(E,\eta))\to \Gamma(X,(\cE,\nabla))\]
  \end{lemma}

  \begin{proof}
     Let $\cE$ be the free sheaf attached to $E_0=E(\Trop(X^\circ))$. Then, a section $s\in\Gamma(\Trop(X^\circ),(E,\eta))$ is an element $s_0\in E_0$ with $\eta s_0=0$ when considered as an element of $(M/\Span(\Trop(X^\circ))^\perp)\otimes E_0$.    
     We wish to show $\theta s_0=0$ where $s_0$ is considered as a constant section of $\cE$. 
    Decompose $\theta s_0=\kappa_s+\kappa_h$ where
     \[\kappa_s\in\Gamma(X,\Omega_{(X,M_s)}^1\otimes \cE),\ 
    \kappa_h\in\Gamma(X,(M''\otimes\cO_X)\otimes \cE).\]
    Then, $\kappa_s$ is regular by the vanishing of its residues, and hence is an element of  $\Gamma(X,\Omega^1_{X})\otimes E_0$. Hence, $\kappa_s=0$. Since $\kappa_h$ vanishes because $\eta s_0=0$, we have $\theta s_0=0$.
     
     This homomorphism is injective by definition. To prove that it is surjective, consider an element of $\Gamma(X,(\cE,\nabla))$ taken to be      
     $s_0\in E_0$ with $\theta s_0=0$. 
     Pick a representative $\eta\in M_\C\otimes\End(E_0)$, so $\eta s_0\in M_C\otimes E_0$. We must show $\eta s_0=0$ as an element of $\Gamma(\Trop(X^\circ),\Omega^1_{(\Delta,N\to N',R)/\Delta_S})\otimes E_0$. 
     Decompose $\theta s_0=\kappa_s+\kappa_h$ as above.
     Similarly, write
     $\eta s_0=\lambda_s+\lambda_h$ for
     \[\lambda_s\in M'_C\otimes E_0,\ \lambda_h\in M''_\C\otimes E_0.\] 
     Since      $\kappa_h=0$, $\lambda_h=0$.
     By Proposition~\ref{p:residueproperties}(\ref{i:transverseresidues}) for all rays $\rho$ of $\Sigma$, $\Res_{D_\rho}(\kappa_s)=\Res_{X_\rho}(\kappa_s)=0.$ Hence, for any $v\in N'_{\R}$ along $\rho$, by Remark~\ref{r:uniqueness1form}, we have the vanishing of the interior product $\iota_v(\eta s)=\iota_v(\kappa_s)=0$.  Since the vectors along the rays of $\Trop(X^\circ)$ span the vector space $\Span_{\R}(\Trop(X^\circ)$, $\kappa_s=0$. Hence $\eta s_0=0$, and thus $s_0$ is a horizontal section of $(E,\eta)$.  
       \end{proof}

\begin{prop} \label{p:equivalence}
  Let $(\Delta,\on{Id}\colon N\to N,R)$ be a trivially enriched fan with partial compactification inducing a log structure on $\PP(\Delta)$.
  Let $X$ be a subspace compactification in $\PP(\Delta)$ defined over $\C$, and let $\Delta$ induce a polyhedral structure on $\Trop(X^\circ)$. Suppose that $R$ is a union of rays of $\Trop(X^\circ)$. Then the functor on unipotent categories
    \[F\colon \sC^{\un}(\Trop(X^\circ),N,R)\to \sC^{\un}((X,M_X)/S)\]
    is essentially surjective, hence an equivalence of categories.
\end{prop}

\begin{proof}
  Because $H^1(X,\cO_X)=0$, the underlying bundle of any object of $\sC^{\un}((X,M_X)/S)$ is trivial. Write the connection $\nabla=d-\theta$ for $\theta\in\Gamma(X,\Omega^1_{(X,M)/S}\otimes \End(\cE))$ satisfying $d\theta-\theta\wedge\theta=0$. By considering $\theta$ as a matrix of log $1$-forms with respect to the standard toric log structure (for which $R=\varnothing$), by Lemma~\ref{l:1formispullback}, $\theta$ is the restriction of a matrix of log $1$-forms on $\PP(\Delta)$, hence it is a
  tropical pullback of a matrix of tropical $1$-forms $\eta$. By taking residues of $\theta\wedge\theta=0$, we see $\eta\wedge\eta=0$. 
  By applying Remark~\ref{r:uniqueness1form}, we see that the interior product of $\eta$ with vectors along the rays in $R$ vanishes, and hence, $\eta$ is a matrix of tropical $1$-forms with respect to  the partial compactification $R$. 
\end{proof}

\subsection{Proof of Proposition~\ref{p:mainmaintheorem}}

\begin{proof}
  Write $\Gamma$ for the simplicial subcomplex of $\Trop(X^\circ)$ consisting of bounded faces. We define two Tannakian categories, $\sC^{\trop}_\Gamma$ and $\sC^{\dR}_\Gamma$ over $\Gamma$.
  For $P\in\Gamma$, set
  \[\sC^{\trop}_P\coloneqq\sC^{\un}((\Trop(X^\circ),N,R)_P),\ \sC^{\dR}_P\coloneqq\sC^{\un}((X_P,M_P)/S^\dagger).\]
  
  Define the natural pullbacks induced by inclusions of faces $P_1\subseteq P_2$ to obtain a category over a simplicial complex and hence descent categories.
   By Proposition~\ref{p:tropicaldescent} and Proposition~\ref{p:schondescent}, there are  equivalences of  categories
  \[\sC^{\dR}(\Gamma)^{\un}\cong \sC^{\un}((X_0,M)/S^\dagger),\ \sC^{\trop}(\Gamma)^{\un}\cong \sC(\Trop(X^\circ),N,R).\]

   Let $\tilde{\Omega}_P$ be the shortened complex on $X_P$  pulled back by $X_P\hookrightarrow \PP(\Sigma_P)$ where $\PP(\Sigma_P)$ is given the shortened complex induced by $(\Sigma_P,N\to N/N_P,R')$. 
   By Lemma~\ref{l:logstructureontransversetoorbits}, the shortened complex of differential forms on $X_P$ induced by $(X_P,M_P)/S^\dagger$ is canonically isomorphic to $\tilde{\Omega}_P$.
   Since each $X_P$ is a subspace compactification by Proposition~\ref{p:matroidtrop}, hence some-kind-of-wonderful, there is a tropical pullback functor 
   \[F_P\colon \sC_P^{\trop}=\sC^{\un}((\Trop(X^\circ),N,R)_P)\to \sC(X_P/S,\tilde{\Omega}_P)\cong \sC_P^{\dR}\]
  which induces a functor of descent categories.
  This functor is fully faithful for all $P$ by Lemma~\ref{l:fullyfaithful}. For each vertex $v$, $F_v$ is an equivalence of categories by Proposition~\ref{p:equivalence}. Therefore, it induces an equivalence of descent categories by Lemma~\ref{l:descentequivalence}.

  Because $H^1(X_P,\cO_{X_P})=0$ as a subspace compactification, $X_0$ is locally topologically acyclic. 
  Let $\omega_{X_0,P}$ be the attached regular sections functor on $\sC^{\dR}_P$ for $P\in\Gamma$. Then, it is easy to see $\omega_{X_0,P}\circ F$ is the fiber functor on $\sC^{\trop}_P$ given by $E\mapsto E(\Star^\circ_P(\sigma))$. Consequently, on descent categories, $\omega_{X_0,P}\circ F=\omega_{\Sigma,P}$. This yields the isomorphism of Tannakian fundamental groups.
\end{proof}


\appendix 

\section{Descent for simple normal crossings varieties} \label{a:appendix}
The aim of this appendix is to prove a descent theorem for labelled normal crossings. Strict normal crossings schemes are labelled normal crossings by, say \cite[Tag 0CAT]{stacks-project}. Background information on normal crossings divisors can be found in \cite[Tag 0CBN]{stacks-project}.

\begin{definition} \label{d:labellednc}
A scheme $ X $ over a field $k$ is {\em normal crossings} if there is an \'{e}tale cover $ \left\{p_\alpha\colon U_{\alpha}\to X
\right\} $ of $X $ and \'{e}tale 
\[\pi_{\alpha}\colon U_{\alpha}\to \Spec k[x_{1}, \dots, x_{n}]/(x_{1} \dots x_{n}).\] 
Write $S_i=\Spec k[x_{1}, \dots, x_{n}]/(x_{i})$.
Let $X=\bigcup_v X_v$ be the decomposition of $X$ into irreducible components.
The variety $X$ is {\em labelled normal crossings} if  there is a cover as above such that for each $U_\alpha$ and each $i$ such that $\pi_{\alpha}^{-1}(S_i)\neq\varnothing$, there exists $v$ such that 
$\pi_{\alpha}^{-1}(S_i)=
p_{\alpha}^{-1}(X_v).$
\end{definition}

The main result of this appendix is the following:

\begin{theorem} \label{t:amainresult}
Let $ X $ be a labelled simple normal crossings scheme. There is an equivalence of categories
\[ \left\{ \begin{tabular}{c}
     Descent data of vector bundles on \\
     the irreducible components of $X$ 
\end{tabular} \right\}  \simeq \left\{ \textrm{Vector 
bundles on } X \right\}. \]
\end{theorem}

By descent arguments with respect to the cover $\{U_{\alpha}\}$, it suffices to prove this result for a single $U_{\alpha}$.

\subsection{Descent for fiber product rings}
We include a brief descent commentary (or dysentery) for fiber product rings. Here, we will extend in a special case, via induction, ideas from \cite{Ferrand} that have been also written up in  \cite[Tag 08KG]{stacks-project}. The original treatment is sufficient for the case of two components where one does not have the cocycle condition.
Let $R$ be a ring with homomorphisms $p_i\colon R\to R_i$ for $i=1,\dots,n$. Write $R_{ij}\coloneqq R_i\otimes_R R_j$ where we identify $R_{ij}$ and $R_{ji}$. We say that $(R,R_i,p_i)$ is a {\em fiber product data for $R$} if the homomorphism 
$R\to R_1\times \dots \times R_n$ induces an isomorphism of $R$ with the subring
\[R_1\times_p \dots\times_p R_n\coloneqq\{(r_1,\dots,r_n)\in R_1\times \dots\times R_n| 1 \otimes r_j=r_i \otimes 1 \textrm{ in } R_{i} \otimes_{R} R_{j}\}.\]
This condition is equivalent to the exactness of the sequence of $R$-modules,
\[\xymatrix{0\ar[r]&R\ar[r]&\prod R_i\ar[r]&\prod_{i<j} R_{ij}
}\]
where the second arrow to $R_{ij}$ is given by equalization along $R_i\times R_j\to R_{ij}$.

\begin{definition}
Write $R=k[x_1,\dots,x_n]/(x_1\dots x_n)$, $R_i=k[x_1,\dots,x_n]/(x_i)$ and $p_i\colon R\to R_i$ the usual quotient. We say that $(S,S_i,q_i)$ is {\em flat over the simple normal crossings model} if $S$ is flat $R$-module such that there is an isomorphism $S_i\cong R_i\otimes_R S$ making $q_i\colon S\to S_i$ equal to $p_i\otimes S$. 
\end{definition}

By the exact sequence characterization of fiber product data, we immediately obtain that $(S,S_i,q_i)$ is fiber product data.

\begin{definition} {\em Descent data} for a module over $(S,S_i,q_i)$ is a collection of modules $M_i$ over $S_i$ together with $S_{ij}$-isomorphisms 
\[\varphi_{ij}\colon M_i\otimes_{S} S_j\to M_j\otimes_{S} S_i\]
with $\varphi_{ii}=\on{Id}_{M_i}$
satisfying the cocycle condition\footnote{Here and throughout we abuse notation and conflate a map
$\phi\colon M \to N $ and its induced map
$M \otimes S' \to N \otimes S' $.}
\[\varphi_{jk}\circ\varphi_{ij}=\varphi_{ik}\colon M_i\otimes_{S} S_{jk}\to M_k\otimes_{S} S_{ij}.\]

Given  descent data, we define the {\em fiber product module}
\[M_1\times_\varphi\dots\times_{\varphi}M_n=\{(m_1,\dots,m_n)| m_i\in M_i, \varphi_{ij}(m_i \otimes 1)=m_j\otimes 1\}.\]
\end{definition}

Let $P$ be one of the following criteria: flat, flat and finite, projective and finite. We will slightly generalize the following result albeit in a special case:
\begin{theorem}[{cf. \cite[{Th\'{e}or\`{e}me 2.2 (iv)}]{Ferrand}}]\label{ferrandThm}
Let $ R \to R_{1} $, $R \to R_{2} $ be a fiber product ring with one of
$ R_{i} \to R_{1} \otimes_{R} R_{2} $ surjective. Let 
$ \mathcal{C}(A) $ denote the category of $A$-modules with $P$.
Then the fiber product module functor yields an equivalence of categories
\[ \mathcal{C}(R_{1}) \times_{\mathcal{C}(R_{1} \otimes_{R} R_{2})} \mathcal{C}(R_{2})
 \simeq \mathcal{C}(R).
\]
\end{theorem}

We will inductively use this result to prove the following:

\begin{theorem}\label{t:SNCAffineDescent}
Let $ (S,S_{i}, q_{i}) $ be  flat over $(R,R_i,p_i)$.
 For a commutative ring $ A $, let $ \mathcal{C}(A) $ denote the categories of finite, locally free $A$-modules.
 There is an equivalence of categories
 \[ \left\{ \textrm{Descent data on } \mathcal{C}(S_{i}) \right\} \simeq
 \mathcal{C}(S) \]
 given by the fiber product module.
\end{theorem}

Write 
\[S_1\times_q \dots\times_q S_{n-1}\coloneqq \{(s_1,\dots,s_{n-1})\in S_1\times \dots\times S_{n-1}| 1 \otimes s_j=s_i \otimes 1 \textrm{ in } S_{i} \otimes_{S} S_{j}\}\]
As tensor products commute with direct sums, $S_i\to S_n\otimes_S S_i$ induces a homomorphism
\[q'_1\colon S_1\times_q\dots\times_q S_{n-1}\to S_n\otimes_{S} (S_1\times_q\dots\times_q S_{n-1}).\]
Together with the natural homomorphism $q'_2\colon S_n\to S_n\otimes_{S} (S_1\times_q\dots\times_q S_{n-1})$, this induces a fiber product
$(S_1\times_p\dots \times_p S_{n-1})\times_{q'} S_n$. The following is a straightforward verification:

\begin{lemma}\label{l:hierarchicalRing} For fiber product data $(S,S_i,q_i)$ flat over the simple normal crossings model, the natural maps induce an isomorphism
\[S_1\times_q\dots \times_q S_n\to (S_1\times_q\dots \times_q S_{n-1})\times_{q'} S_n.\]
\end{lemma}

In our context of fiber product data flat over the simple normal crossings model, we are able 
to build the fiber product module factor by factor.

\begin{lemma}\label{inductGluingData}
Suppose Theorem~\ref{t:SNCAffineDescent} is true for $n-1$.
Let $(S, S_{i}, q_{i}) $ be fiber product ring data flat over the simple normal crossings model, and let $ (M_{i}, \varphi_{ij}) $ descent data
over $ (S, S_{i}, q_{i}) $ with each $M_{i} $ finite and locally free. 
Then  $(\varphi_{ij}) $ induce an
isomorphism
\[ \psi\colon M_{n} \otimes_S \left( S_{1} \times_{p} \dots
\times_{p} S_{n-1} \right) \to \left(M_{1} \times_{\varphi} \dots \times_{\varphi} M_{n-1}
\right) \otimes_S S_{n}.\]
\end{lemma}

\begin{proof}
Write $F_{n-1} = S_{1} \times_{q} \dots \times_{q} S_{n-1}.$
We will express $\psi$ as the composition of three isomorphisms. First, observe that $S_n\otimes_S F_{n-1}$ is flat over the simple normal crossings model $K[x_1,\dots,x_{n-1}]/(x_1\dots x_{n-1})$.
Now, note that the $(S_n\otimes_S F_{n-1})$-module $M_n\otimes_S F_{n-1}$ arises from the descent data 
\[(M_n\otimes_S S_i,\ \on{Id}_{ij}\colon (M_n\otimes_S S_i)\otimes_S S_j\to (M_n\otimes_S S_j)\otimes_S S_i),\]
giving an isomorphism
\[
 \psi_{1}\colon M_{n} \otimes_S \left( S_{1} \times_{q} \dots \times_{q} S_{n-1} \right)
 \to \left(M_n \otimes_S S_{1} \right) \times_{\on{Id}} \dots \times_{\on{Id}} \left( M_{n} \otimes_S
 S_{n-1} \right).
\]
Indeed, this is evidently true for free modules, hence true for locally free ones by checking every prime.

Our second isomorphism is
\[ \psi_{2}\colon \left(M_{n} \otimes_S S_{1}\right) \times_{\on{Id}} \dots \times_{\on{Id}}
(M_{n} \otimes_S S_{n-1}) \to \left(M_{1} \otimes_S S_{n}\right)
\times_{\varphi} \dots \times_{\varphi} \left( M_{n-1} \otimes_S S_{n} \right)\]
induced on the $i$th coordinate by $ \varphi_{ni} $.
This is well-defined because following diagram commutes  
\[\xymatrix{
( M_{n} \otimes_S S_{1} ) \times_{p} 
    \dots \times_{p} (M_{n} \otimes_S S_{n-1} ) 
    \ar[r] & \prod\limits_{i=1}^{n-1} ( M_{n} \otimes_S S_{i} ) \ar[d]^{\varphi_{ni}}\ar[r]&
    \prod\limits_{i<j<n} M_{n} \otimes_S S_{ij} \ar[d]^{\varphi_{ni}\otimes S_j} \\
    \left( M_{1} \otimes_S S_{n} \right) \times_{\varphi} \dots
    \times_{\varphi} ( M_{n-1} \otimes_S S_{n} ) \ar[r] &
    \prod\limits_{i=1}^{n-1} ( M_{i} \otimes_S S_{n} ) \ar[r] &
    \prod\limits_{i < j < n} M_{i} \otimes_S S_{nj} 
}\]
which one verifies using the cocycle condition.

For the final isomorphism,
\[
\psi_{3} \colon \left( M_{1} \otimes_S S_{n} \right)\times_{\varphi} 
\dots \times_{\varphi} \left( M_{n-1} \otimes_S S_{n} \right) \to \left( M_{1} 
\times_{\varphi} \dots \times_{\varphi} M_{n-1} \right) \otimes_S S_{n},
\]
consider  the 
exact sequence arising from realizing the locally free $F_{n-1}$-module $M_1\times_{\varphi}\dots\times_{\varphi} M_{n-1}$ by descent,
\[ 0 \to M_{1} \times_{\varphi} \dots \times_{\varphi} M_{n-1} \to
M_{1} \times \dots \times M_{n-1} \to \prod_{i<j<n} M_{i} \otimes_{F_{n-1}} {S_{j}}.\]
Tensor the above exact sequence of $F_{n-1}$-modules with $S_n$ over $S$, observing that there is no $x_n$-torsion, and so it remains exact. Hence, it exhibits 
$\left( M_{1} 
\times_{\varphi} \dots \times_{\varphi} M_{n-1} \right) \otimes_S S_{n}$
as arising by descent from $(M_i\otimes_{F_n} S_n)_{i=1}^{n-1}$.
\end{proof}

The following is a straightforward verification using the above isomorphism:

\begin{lemma}\label{descentFlatModel}
Let $ (S, S_{i}, q_{i}) $ be fiber product ring data flat over the simple normal crossings
model. Let $ (M_{i}, \varphi_{ij}) $ be descent data with each $M_{i} $ finite and 
locally free. Then for each $i$,
\[ M_{1} \times_{\varphi} \dots \times_{\varphi} M_{i} = 
\left( M_{1} \times_{\varphi} \dots \times_{\varphi} M_{i-1} \right) \times_{\psi} M_{i}.
\]
\end{lemma}

%
%

Now we prove Theorem~\ref{t:SNCAffineDescent}.

\begin{proof}
We will induct on $ n $, the number of components in the  normal crossings model. 
The base case of $n=2 $ was stated above. 

Suppose the theorem has been proven for $ n-1 $ components.
Then, given fiber product data $ (S, S_1,\dots,S_n,q_1, \dots,q_n)$ and
descent data on $ \mathcal{C}(S_{i}) $, the first $n-1$ pieces define 
descent data for $(F_{n-1},S_1,\dots,S_{n-1},q_1,\dots,q_{n-1})$.
By the inductive hypothesis, 
$ M_{1} \times_{\varphi} \dots \times_{\varphi} M_{n-1} $ is a finite, locally free 
$F_{n-1}$-module. By Lemma \ref{inductGluingData},
the $ \varphi_{ij} $ induce descent data in \[ \mathcal{C}(S_{1} \times_{q} \dots \times_{q}
S_{n-1}) \otimes_{\mathcal{C}(S_{1} \times_{q} \dots \times_{q} S_{n-1} \otimes_{S}
S_{n})} \mathcal{C}(S_{n}).\]
By Theorem \ref{ferrandThm} and Lemma \ref{descentFlatModel}, 
\[ \left( M_{1} \times_{\varphi} \dots \times_{\varphi} M_{n-1} \right)\times_{\psi}
M_{n} = M_{1} \times_{\varphi} \dots \times_{\varphi} M_{n} \]
is a finite, locally free $ S $-module. So we have a functor
\[ T\colon \left\{ \textrm{Descent data on } \mathcal{C}(S_{i}) \right\}  \simeq
\mathcal{C}(S) \]

We now construct an inverse functor.
Given an $ S $ module $M$, we get descent data by setting $ M_{i} = M \otimes_{S} S_{i} $, and $ \varphi_{ij} $ the natural map $ M_{i} \otimes_{S} S_{j} \to
M_{j} \otimes_{S} S_{i} $.
Using Theorem \ref{ferrandThm} on $ S_{1} \times_{q} \dots \times_{q}
\widehat{S_{i}} \times_{q} \dots \times_{q} S_{n} $ and $ S_{i} $,
we see that each $ M_{i} $ is finite and locally free.

One immediately verifies, using Theorem~\ref{ferrandThm}, that these functors
are inverses of one another.
\end{proof}

\bibliographystyle{plain}
\bibliography{references}

\end{document}